\documentclass[11pt,a4paper,reqno]{amsart}
\pdfoutput=1
\usepackage[english]{babel}
\usepackage{amssymb,xcolor,subfig}
\usepackage[pdfstartview=FitH,urlbordercolor=blue!50!cyan]{hyperref}
\usepackage[top=2.5cm,right=2.6cm,left=2.6cm,bottom=2cm]{geometry}
\usepackage{tikz,mathtools}
\usetikzlibrary{arrows,positioning}

\title{Topics in noncommutative geometry}

\date{October 2015}

\author[F.~D'Andrea]{Francesco D'Andrea}
\address{Dipartimento di Matematica e Applicazioni, Universit\`a di Napoli ``Federico II''
         and I.N.F.N. Sezione di Napoli, Complesso MSA, Via Cintia, 80126 Napoli, Italy}
\email{francesco.dandrea@unina.it}


\allowdisplaybreaks[2]
\numberwithin{equation}{section}
\newtheorem{thm}{Theorem}[section]
\newtheorem{lemma}[thm]{Lemma}
\newtheorem{prop}[thm]{Proposition}

\newtheorem{df}[thm]{Definition}
\newtheorem{ex}[thm]{Example}
\newtheorem{es}[thm]{Exercise}

\theoremstyle{definition}
\newtheorem{rem}[thm]{Remark}

\newcommand{\arxiv}[1]{arXiv:\htmladdnormallink{#1}{http://arxiv.org/abs/#1}}
\newcommand{\doi}[1]{doi:\htmladdnormallink{#1}{http://dx.doi.org/#1}}

\newcommand{\A}{\mathcal{A}}
\newcommand{\B}{\mathcal{B}}

\newcommand{\U}{\mathcal{U}}
\newcommand{\Q}{\mathbb{Q}}
\newcommand{\R}{\mathbb{R}}
\newcommand{\C}{\mathbb{C}}
\newcommand{\N}{\mathbb{N}}
\newcommand{\Z}{\mathbb{Z}}
\newcommand{\T}{\mathbb{T}}
\newcommand{\WP}{\mathbb{P}}
\newcommand{\id}{\mathsf{id}}
\newcommand{\mc}{\mathcal}
\newcommand{\mf}{\mathfrak}
\newcommand{\mr}{\mathrm}
\newcommand{\mv}{\underline}
\newcommand{\botimes}{\;\bar\otimes\;}
\newcommand{\inner}[1]{\left<\smash[t]{#1}\right>}
\newcommand{\de}{\mathrm{d}}
\newcommand{\quadratini}{{\tiny\raisebox{1pt}{$\blacksquare$}}\,}
\newcommand{\az}{\triangleright}
\newcommand{\lbrak}{[\hspace{-0.5pt}[}
\newcommand{\rbrak}{]\hspace{-0.5pt}]}

\begin{document}

\begin{abstract}
Lecture notes for the autumn school ``From Poisson Geometry to Quantum Fields on Noncommutative Spaces'', University of W{\"u}rzburg, 5-10 October 2015.
\end{abstract}

\maketitle

\makeatletter
\def\@tocline#1#2#3#4#5#6#7{\relax
  \ifnum #1>\c@tocdepth 
  \else
    \par \addpenalty\@secpenalty\addvspace{#2}%
    \begingroup \hyphenpenalty\@M
    \@ifempty{#4}{%
      \@tempdima\csname r@tocindent\number#1\endcsname\relax
    }{%
      \@tempdima#4\relax
    }%
    \parindent\z@ \leftskip#3\relax \advance\leftskip\@tempdima\relax
    \rightskip\@pnumwidth plus4em \parfillskip-\@pnumwidth
    #5\leavevmode\hskip-\@tempdima
      \ifcase #1
       \or\or \hskip 1em \or \hskip 2em \else \hskip 3em \fi%
      #6 \hskip 0.5em \nobreak\relax
    \dotfill\hbox to\@pnumwidth{\@tocpagenum{#7}}\par
    \nobreak
    \endgroup
  \fi}
\makeatother

\setcounter{tocdepth}{2}
\begin{center}
\begin{minipage}{0.8\textwidth}
{\small\tableofcontents}
\end{minipage}
\end{center}

\vspace{5mm}

\thispagestyle{empty}


\section*{Introduction}
The leitmotiv of these lectures is noncommutative principal $U(1)$-bundles and associated line bundles. The only prerequisite is some linear algebra and a basic knowledge of $C^*$-algebras.

In the first part of the minicourse I will give a brief introduction to Hopf-Galois theory and its applications, from field extensions to principal group actions. I will then recall Woronowicz' definition of compact quantum group and the notion of noncommutative principal bundle. When the structure group is $U(1)$, there is a construction due to Pimsner \cite{Pim97} (see also \cite{AKL14}) that allows to get the total space of a ``bundle'' (more precisely, a strongly graded $C^*$-algebra) from the base space and a noncommutative ``line bundle'' (a self-Morita equivalence bimodule). As an example of this construction, I will discuss the $U(1)$-principal bundles of quantum lens spaces over quantum weighted projective space, from \cite{DL14,ADL15}.

The second part is a peek into the realm of nonassociative geometry \cite{BM05,BM10,BSS15a,BSS15b}, which recently raised some interest due to possible applications to the Standard Model of elementary particles \cite{FB13,FB14a,FB14b} and strings \cite{BSS15a}. 
After a review of some properties of Hopf cochains and cocycles,
I~will discuss the theory of cochain quantization and its applications, from Albuquerque-Majid example of octonions \cite{AM99,AM98}, to ``line bundles'' on the noncommutative torus \cite{DF15}.


\section{Hopf-Galois extensions and Morita equivalence}\label{sec:1}

\subsection{Hopf algebras}
Let $\Bbbk$ be a field. By $\Bbbk$-\emph{algebra} I will always mean a unital associative algebra over $\Bbbk$ (unless stated otherwise). Let us recall the definition of $\Bbbk$-coalgebra, obtained from that of $\Bbbk$-algebra by ``reversing the arrows'': this is a $\Bbbk$-vector space $\mc{C}$ with two linear maps $\Delta:\mc{C}\to\mc{C}\otimes\mc{C}$ (coproduct) and $\epsilon:\mc{C}\to\Bbbk$ (counit) satisfying
\begin{align*}
(\id_{\mc{C}}\otimes\Delta)\Delta &=
(\Delta\otimes\id_{\mc{C}})\Delta \;, \qquad\text{(coassociativity)} \\
(\id_{\mc{C}}\otimes\epsilon)\Delta &=
(\epsilon\otimes\id_{\mc{C}})\Delta=\id_{\mc{C}} \;,
\end{align*}
where $\otimes$ is the algebraic tensor product over $\Bbbk$, and we identify $\mc{C}\otimes\Bbbk$ and
$\Bbbk\otimes\mc{C}$ with $\mc{C}$.

Let $\mathsf{Hom}_{\Bbbk}(V,W)$ be the set of linear maps between two $\Bbbk$-vector spaces $V$ and $W$. If $\mc{C}$ is a coalgebra and $\mc{A}$ an algebra, with multiplication map $m$, then $\mathsf{Hom}_{\Bbbk}(\mc{C},\A)$ is a (unital associative) algebra w.r.t.~the convolution product
$$
f_1\star f_2:=m(f_1\otimes f_2)\Delta \;,\qquad
\forall\;f_1,f_2\in\mathsf{Hom}_{\Bbbk}(\mc{C},\A) \;.
$$
In particular the dual $\mc{C}'=\mathsf{Hom}_{\Bbbk}(\mc{C},\Bbbk)$ of a coalgebra is an algebra (the converse statement being not always true in the infinite dimensional case). A \emph{bialgebra} $(\mc{B},m,\Delta,\epsilon)$ is an algebra and a coalgebra s.t.~$\Delta,\epsilon$ are (unital) algebra-morphisms.

If $\mc{B}$ is a bialgebra, then $\mathsf{End}_{\Bbbk}(\mc{B})$ is an algebra with convolution product and with unit element given by the endomorphism $b\mapsto\epsilon(b)1$. We call $\mc{B}$ a \emph{Hopf algebra} if the identity endomorphism $\id_{\mc{B}}$ is invertible in $\mathsf{End}_{\Bbbk}(\mc{B})$ (for the convolution product); its inverse $S:\mc{B}\to\mc{B}$ will be called \emph{antipode} or \emph{coinverse}, and when it exists by construction is also unique.

\begin{es}
(i) Prove that $S$ is a unital and antimultiplicative map, i.e.~$S(1)=1$ and $S(ab)=S(b)S(a)\;\forall\;a,b$.
(ii) Prove that if $\mc{B}$ is a commutative or cocommutative Hopf algebra (i.e.~$\Delta$ is symmetric), then $S^2=\id_{\mc{B}}$.
\end{es}

\begin{ex}
Let $T(V)$ be the free algebra generated by a vector space $V$. This is a Hopf algebra with operations  $\Delta,\epsilon,S$ uniquely defined by
$$
\Delta(v)=v\otimes 1+1\otimes v \;,\qquad
\epsilon(v)=0 \;,\qquad
S(v)=-v \;,
$$
for all $v\in V$. If $V=\mf{g}$ is a Lie algebra and $\mc{U}(\mf{g}):=T(\mf{g})/\inner{xy-yx-[x,y]:x,y\in\mf{g}}$ the universal enveloping algebra, then
$\mc{U}(\mf{g})$ is a Hopf algebra with operations induced by $T(\mf{g})$.
\end{ex}

\begin{ex}
There is an alternative Hopf algebra structure on $T(V)$, given by the shuffle product and deconcatenation coproduct. Recall that a permutation $p\in S_n$ is a $(k,n-k)$ \emph{shuffle} if
$p(1)<p(2)<\ldots<p(k)$ and $p(k+1)<p(k+1)<\ldots<p(n)$. The shuffle product on $T(V)$ is defined on monomials by
$$
(v_1v_2\ldots v_k)\cdot
(v_{k+1}v_{k+2}\ldots v_n)=\sum_{\sigma^{-1}\in S_{k,n-k}}v_{\sigma(1)}v_{\sigma(2)}\ldots v_{\sigma(n)} \;,
$$
where $S_{k,n-k}$ is the set of $(k,n-k)$ shuffles.
The \emph{deconcatenation coproduct} is defined by
$$
\Delta(v_1v_2\ldots v_k)=
1\otimes (v_1v_2\ldots v_k)+\sum_{i=1}^{k-1}
(v_1\ldots v_i)\otimes (v_{i+1}\ldots v_k)+
(v_1v_2\ldots v_k)\otimes 1 \;. \vspace{-20pt}
$$
\end{ex}

\begin{ex}\label{ex:1.4}
Let $G$ be a group.
The \emph{group algebra} $\Bbbk G$ is a vector space with basis $\{g|g\in G\}$ and product extending bilinearly the one of $G$.
It is a Hopf algebra with (co)operations $\Delta,\epsilon,S$ uniquely defined by
$$
\Delta(g)=g\otimes g \;,\qquad
\epsilon(g)=1 \;,\qquad
S(g)=g^{-1} \;,
$$
for all $g\in G$.
\end{ex}

\begin{ex}
Let $G$ be a \emph{finite} group.
Since the space $\Bbbk^G$ of functions $G\to\Bbbk$ is finite dimensional, we can identify
$\Bbbk^{G\times G}$ with $\Bbbk^G\otimes\Bbbk^G$.
Then, $\Bbbk^G$ is a Hopf algebra with pointwise product and with $\Delta,\epsilon,S$ dual to the group operations of $G$:
$$
\Delta f(g,h)=f(gh) \;,\qquad
\epsilon(f)=f(e) \;,\qquad
S f(g)=f(g^{-1}) \;,
$$
for all $f\in\Bbbk^G$, $g,h\in G$ and with $e\in G$ the unit element.
\end{ex}

\medskip

The examples above are either commutative or cocommutative (i.e.~the coproduct maps into the set of symmetric tensors). A class of finite-dimensional Hopf algebras that are neither commutative nor cocommutative was introduced by Taft \cite{Taf71}. (For an infinite-dimensional example, see Example \ref{ex:SUqn}.)

\begin{ex}
Suppose $\Bbbk$ is a field with a primitive $p$-th root of unity
$\lambda$ for some integer $p\geq 2$.
The Taft algebra $T_p(\lambda)$ is the abstract algebra (of dimension $p^2$) generated by $g$ and $x$ with relations
$x^n=0$, $g^n=1$, $gx=\lambda xg$. It is a Hopf algebra with $\Delta,\epsilon,S$ defined by
$$
\Delta(g)=g\otimes g \;,\;\;
\Delta(x)=x\otimes g+1\otimes x \;,\;\;
\epsilon(g)=1 \;,\;\;
\epsilon(x)=0 \;,\;\;
S(g)=g^{-1} \;,\;\;
S(x)=-xg^{-1} \;.
$$
The $p=2$ example is due to Sweedler \cite{Swe69}.
\end{ex}

\medskip

The importance of last example is that (over an algebraically closed field of characteristic zero) every Hopf algebra of dimension $p^2$, with $p$ prime, is either isomorphic to a group algebra $\Bbbk G$ (with $G=\Z_{p^2}$ or $G=\Z_p\times\Z_p$) or to a Taft algebra.

The restricted quantum enveloping algebra $U_q^{\mathrm{res}}(\mf{sl}_2)$ (which I shall not discuss here), for $q^2$ a root of unity, has Borel subalgebras isomorphic to Taft algebras \cite{CP94,Mon09}.

\subsection{Hopf-Galois theory}
Let $H$ be a coalgebra. A right $H$-\emph{comodule} is
a vector space $A$ with a linear map $\delta:A\to A\otimes H$ (the \emph{coaction}) satisfying
$$
(\delta\otimes\id)\delta=(\id\otimes\Delta)\delta \;,\qquad
(\id\otimes\epsilon)\delta=\id \;.
$$
We denote by $A^{\mathrm{co}H}:=\big\{a\in A:\delta(a)=a\otimes 1\big\}$ the set of coinvariant elements.

If $H$ is a Hopf algebra, a right $H$-\emph{comodule algebra} is a right $H$-comodule $A$ which is also an algebra, and such that $\delta$ is a homomorphism of unital algebras; in this case, $A^{\mathrm{co}H}$ is a unital subalgebra of $A$. Left comodules are defined in a similar way.

\begin{ex}[Pareigis Hopf algebra \cite{Par81}]
Let $H$ be generated by $x,g,g^{-1}$ with relations $xg+gx=x^2=0$.
It is a Hopf algebra with co-operations:
$$
\Delta(g)=g\otimes g \;,\;\;
\Delta(x)=x\otimes g+1\otimes x \;,\;\;
\epsilon(g)=1 \;,\;\;
\epsilon(x)=0 \;,\;\;
S(g)=g^{-1} \;,\;\;
S(x)=-xg^{-1} \;.
$$
(Note that the Sweedler Hopf algebra is a quotient by the further relation $g^2=1$.) Comodules for this Hopf algebra are the same as chain complexes. Given indeed a chain complex (of $\Bbbk$-vector spaces):
$$
\ldots\stackrel{d}{\longleftarrow}A_{-2}\stackrel{d}{\longleftarrow}A_{-1}\stackrel{d}{\longleftarrow}A_0\stackrel{d}{\longleftarrow}A_1\stackrel{d}{\longleftarrow}A_2\stackrel{d}{\longleftarrow}\ldots
$$
on $A=\bigoplus_{n\in\Z}A_n$ there is a $H$-coaction, defined on homogeneous elements $a\in A_n$ by
\begin{equation}\label{eq:Pareigis}
\delta(a)=a\otimes g^n+da\otimes g^{n-1}x \;.
\end{equation}
This correspondence between $H$-comodules and $\Bbbk$-complexes is an equivalence of categories.
\end{ex}

\begin{es}
(i) Check that \eqref{eq:Pareigis} is indeed a coaction.
(ii) Given a comodule $B$ for the Pareigis Hopf algebra,
explain how to get a chain complex from $B$
(hint: $\{g^n,g^nx\}_{n\in\Z}$ is a basis of $H$, hence
$B\otimes H=\bigoplus_{n\in\Z}B\otimes\Bbbk[g^n]+B\otimes\Bbbk[g^nx]$; call $A_n$ the first summand\hspace{1pt}...).
\end{es}

\begin{df}
Let $A$ be a right $H$-comodule algebra and $B:=A^{\mathrm{co}H}$.
The extension $B\subset A$ is \emph{Hopf-Galois} if the canonical map
\begin{equation}\label{eq:can}
A\otimes_BA\to A\otimes H \;,\qquad
a\otimes_B b\mapsto (a\otimes 1)\delta(b) \;,\vspace{-5pt}
\end{equation}
is bijective \cite{Sch94}.
\end{df}

\noindent
Motivating examples for previous definition come from number theory and differential geometry.

\begin{ex}[Principal group action]\label{ex:principal}
Let $G$ be a finite group and $X$ a finite $G$-space. 
A coaction of $H:=\Bbbk^G$ on $A:=\Bbbk^X$ is given by
$\delta f(x,g):=f(xg)$
(where we identify $\Bbbk^{X\times G}$ with $\Bbbk^X\otimes\Bbbk^G$).
The action of $G$ on $X$ is free and transitive if{}f the map
$$
X\times G\to X\times_{X/G}X \;,\qquad
(x,g)\mapsto (x,xg) \;,
$$
is bijective,
which is equivalent (dual) to bijectivity of the canonical map \eqref{eq:can}.
\end{ex}

\medskip

For topological groups, the appropriate framework is that of $C^*$-algebras, which is recalled in the next section.

\begin{ex}[Galois extension]
Let $E\subset F$ be two fields, $G=\mathrm{Aut}_E(F)$ the group of automorphisms of $F$ fixing the elements in $E$.
Recall that $E\subset F$ is called a \emph{Galois extension} if 
$F$ is algebraic over $E$ and all $G$-invariant elements of $F$ belongs to $E$. 

Let $\Bbbk\subset F$ be a subfield, $G$ a finite group acting as $\Bbbk$-automorphisms of $F$, and $H:=\Bbbk^G$. A coaction of $H$ on $F$ is given by
$$
\delta:F\to F\otimes_{\Bbbk}H\simeq F^G
\;,\qquad
\delta f(g):=g(f) \;.
$$
The extension $F^{\hspace{1pt}\mathrm{co}H}\subset F$ is Galois if{}f it is Hopf-Galois \cite{Mon09}.
\end{ex}

\smallskip

There are extensions that are not Galois for any group $G$, but are Hopf-Galois for some finite-dimensional Hopf algebra $H$; moreover, while for a Galois extension $G$ is unique, the same extension can be Hopf-Galois with respect to different (non isomorphic) Hopf algebras. Some properties, like the existence of a \emph{normal basis}, can be generalized to the Hopf-Galois setting \cite{Mon09}.

\begin{rem}\label{rem:prin}
In the spirit of Example \ref{ex:principal}, one could think of a Hopf-Galois extensions as noncommutative generalization of a principal bundle. In fact, several alternative definitions have been proposed, that allow to reproduce known properties of principal bundles.

In \cite{BM92}, the authors give a definition of \emph{quantum principal bundle} which depends on the choice of a differential calculus, and includes Hopf-Galois extensions as a subclass of examples.
In \cite{Haj96}, it was shown that a Hopf-Galois extension is the same as a quantum principal bundle in the sense of \cite{BM92} with universal differential calculus.
(In \cite{BM92} it is assumed that the ground field has characteristic $\neq 2$, and in \cite{Haj96} that the characteristic is $0$.)

In \cite{Dur96,Dur97} there is an alternative definition of quantum principal bundle, very similar to the one in \cite{BM92} although not completely equivalent. It was shown in \cite{Dur97} that, for a compact matrix quantum group (a compact quantum group generated by the matrix entries of a finite-dimensional unitary corepresentation, cf.~\S\ref{sec:cqg}), this definition is equivalent to that of Hopf-Galois extension.

Other generalizations of the notion of principal bundle are possible. It is worth mentioning \emph{principal comodule algebras}, that are Hopf-Galois extensions satisfying an additional condition called ``equivariant projectivity'' \cite{HKMZ07},
and \emph{principal extensions}, that are principal comodule algebras satisfying two additional conditions \cite{BH04}. 
\end{rem}

\subsection{Compact quantum groups}\label{sec:cqg}
Originally introduced in \cite{Wor87,Wor95}. I will adopt the terminology of \cite{Wan98}. Here $\botimes$ is the the minimal tensor product of $C^*$-algebras.

\begin{df}\label{def:woro}
A \emph{Woronowicz $C^*$-algebra} is a pair $(Q,\Delta)$ given by a complex unital \mbox{$C^*$-algebra} $Q$ and a unital $C^*$-algebra morphism $\Delta:Q\to Q\botimes Q$
such that
\begin{itemize}\itemsep=2pt
\item[i)]
$\Delta$ is coassociative, i.e.
$$
(\Delta\otimes\id)\circ\Delta=(\id\otimes\Delta)\circ\Delta
$$
as equality of maps $Q\to Q\botimes Q\botimes Q$;
\item[ii)] the sets
$\mathrm{Span}\bigl\{(a\otimes 1)\Delta(b)\,\big|\,a,b\!\in\! Q\bigr\}$ and
$\,\mathrm{Span}\bigl\{(1\otimes a)\Delta(b)\,\big|\,a,b\!\in\! Q\bigr\}$
are norm-dense in $Q\botimes Q$.
\end{itemize}
\end{df}

\noindent
Any commutative Woronowicz $C^*$-algebra is of the form $Q=C(G)$, the algebra of continuous functions on a compact topological group $G$, with pointwise product, sup norm and with $\Delta(f)(x,y)=f(xy)$ the pullback of the group multiplication.
For $Q=C(G)$, conditions i) and ii) correspond to the associativity
and cancellation property of the product in $G$, respectively.

Woronowicz $C^*$-algebras form a category, with morphisms given by
$C^*$-homomorphisms intertwining the coproduct. Given a Woronowicz
$C^*$-algebra $(Q,\Delta)$, we think of it as describing a \emph{compact quantum group}, defined as the dual object in the opposite category.

\begin{df}
A $n$-dimensional corepresentation of $Q$
is a matrix $U=(u^i_j)\in M_n(Q)$ s.t.
$$
\Delta(u^i_j)=\sum\nolimits_ku^i_k\otimes u^k_j \;.
$$
The corepresentation is:
\emph{faithful} if the set $\{u^i_j\}$ generates $Q$;
\emph{unitary} if \ $UU^*=U^*U=1\,$.
\end{df}

Unitary representations can be infinite-dimensional as well, but for the sake of simplicity I omit the definition.

For any Woronowicz $C^*$-algebra $Q$, there always exists a dense $*$-subalgebra $Q_0$ which is a Hopf $*$-algebra with the same coproduct $\Delta$. This subalgebra is canonical, and spanned by the matrix coefficients of the finite dimensional unitary corepresentations of $Q$ \cite{Wor87,Wor95}.

\begin{df}\label{def:alpha}
A \emph{coaction} of a Woronowicz $C^*$-algebra on a unital $C^*$-algebra $A$ is a unital $C^*$-homomorphism $\alpha: A\to A\botimes Q$ such that:
\begin{itemize}\itemsep=3pt
\item[i)] $(\alpha \otimes\id) \alpha=(\id\otimes \Delta) \alpha$,
\item[ii)]
$\mathrm{Span}\bigl\{\alpha(b)(1_\B \otimes a)
\,\big|\,b\in A,\,a\in Q\bigr\}$ is norm-dense in $A\botimes Q$.
\end{itemize}
\end{df}

\noindent
Condition (ii) in Def.~\ref{def:alpha} is equivalent to the existence of a norm-dense $*$-subalgebra $A_0$ of $A$ such that $\alpha$ restricted to $A_0$ is a coaction of the Hopf $*$-algebra $Q_0$~\cite{Pod95,Wan98}.
We think of $A$ as the algebra of continuous functions on a virtual compact quantum homogeneous space where the compact quantum group acts.

\medskip

With this machinery, one can generalize several geometrical notions, like the one of principal bundle. 
One could give a definition analogous to the one of Hopf-Galois extension in the $C^*$-algebra setting, but it is easier to consider a coaction of $Q$ on $A$ ``principal'' when (in the notation above)
$A_0^{\mathrm{co}\hspace{1pt}Q_0}\subset A_0$ is Hopf-Galois (or it is principal according to one of the definitions in Rem.~\ref{rem:prin}).
One can notice that when discussing topological properties of extensions, such as the one of being \emph{piecewise trivial} \cite{HKMZ07}, both $(Q,A)$ and $(Q_0,A_0)$ play a role, and one cannot avoid using $C^*$-algebras (see e.g.~the discussion in the introduction of \cite{HKMZ07}).

\begin{ex}
Let $G$ be a discrete group. The reduced group $C^*$-algebra $C^*_r(G)$ is a Woronowicz $C^*$-algebra with operations analogous to that of Example \ref{ex:1.4}. Note that if $G$ is finite, then $C^*_r(G)=\C G$ is just the group algebra of Example \ref{ex:1.4}.
\end{ex}

\begin{ex}
The $C^*$-algebra $A_s(n)$ of the \emph{quantum permutation group} $S^+_n$ is defined as the universal Woronowicz $C^*$-algebra with a coaction on the set with $n$ elements \cite{Wan98}. Its abelianization $C(S_n)$ is the algebra of continuous functions on the ordinary group of permutation of $n$ elements, and in fact $S_n^+=S_n$ if $n\leq 3$. If $n>3$, $S_n\subset S_n^+$ is a proper quantum subgroup, being $A_s(n)$ non-commutative and infinite-dimensional.
\end{ex}

\begin{ex}\label{ex:SUqn}
Let $0<q\leq 1$.
The Hopf $*$-algebra
$\mc{O}(SU_q(n+1))$ of ``representative functions'' on the compact quantum group $SU_q(n+1)$ is defined as follows.
It is generated by the matrix elements of an $n+1$-dimensional corepresentation $U=(u^i_j)_{i,j=1,...,n+1}$, with commutation relations \cite[Sec.~9.2]{KS97}:
\begin{align*}
u^i_ku^j_k &=qu^j_ku^i_k &
u^k_iu^k_j &=qu^k_ju^k_i &&
\forall\;i<j\;, \\
[u^i_l,u^j_k]&=0 &
[u^i_k,u^j_l]&=(q-q^{-1})u^i_lu^j_k &&
\forall\;i<j,\;k<l\;,
\end{align*}
and with determinant relation
$$
\sum\nolimits_{\pi\in S_{n+1}}(-q)^{\|\pi\|}
u^1_{\pi(1)}u^2_{\pi(2)}\ldots u^{n+1}_{\pi(n+1)}=1 \;,
$$
where the sum is over all permutations $\pi$ of the set
$\{1,2,\ldots,n+1\}$ and $\|\pi\|$ is the number of inversions in $\pi$. The $*$-structure is given by
$$
(u^i_j)^*=(-q)^{j-i}\sum\nolimits_{\pi\in S_n}(-q)^{\|\pi\|}u^{k_1}_{\pi(l_1)}
u^{k_2}_{\pi(l_2)}\ldots u^{k_n}_{\pi(l_n)}
$$
with $\{k_1,\ldots,k_n\}=\{1,\ldots,n+1\}\smallsetminus\{i\}$,
$\{l_1,\ldots,l_n\}=\{1,\ldots,n+1\}\smallsetminus\{j\}$
(as ordered sets) and the sum is over all permutations $\pi$ of $n$ elements.
Coproduct, counit and antipode are of `matrix' type:
$$
\Delta(u^i_j)=\sum\nolimits_ku^i_k\otimes u^k_j\;,\qquad
\epsilon(u^i_j)=\delta^i_j\;,\qquad
S(u^i_j)=(u^j_i)^*\;.
$$
The associated universal $C^*$-algebra is a Woronowicz $C^*$-algebra in the sense of Def.~\ref{def:woro}, denoted $C(SU_q(n+1))$.
%
\end{ex}

\begin{ex}\label{ex:Sqn}
The algebra $\mc{O}(S^{2n+1}_q)$ of `functions' on the unitary quantum sphere $S^{2n+1}_q$ is the $*$-algebra generated by
$\{z_i,z_i^*\}_{i=0,\ldots,n}$ with relations \cite{VS91}:
\begin{align*}
z_iz_j &=q^{-1}z_jz_i \hspace{-5mm} &&\forall\;i<j \;,&&
[z_i^*,z_i] =(1-q^2)\sum\nolimits_{j=i+1}^n z_jz_j^* 
    \quad\forall\;i\neq n \;,\\[2pt]
z_i^*z_j &=qz_jz_i^*  \hspace{-5mm} &&\forall\;i\neq j \;, &&
[z_n^*,z_n] =0 \;,
\end{align*}
\begin{center}
$z_0z_0^*+z_1z_1^*+\ldots+z_nz_n^*=1 \;.$\vspace{5pt}
\end{center}
An injective $*$-algebra morphism
$\imath:\mc{O}(S^{2n+1}_q)\hookrightarrow\mc{O}(SU_q(n+1))$ 
is given by the map $z_i\mapsto u_{n+1-i}^{n+1}$.
We will identify $\mc{O}(S^{2n+1}_q)$ with the corresponding
subalgebra of $\mc{O}(SU_q(n+1))$, and define $C(S^{2n+1}_q)$ as its $C^*$-completion in $C(SU_q(n+1))$.

The restriction of the coproduct to $\mc{O}(S^{2n+1}_q)$ turns it into a right $\mc{O}(SU_q(n+1))$-comodule algebra.
Denoting by $\{\tilde u^i_j\}_{i,j=1}^n$ the generators of $\mc{O}(SU_q(n))$, a left coaction on $\mc{O}(SU_q(n+1))$ is given on generators by $\delta(u^i_j)=\sum_{k=1}^n\tilde u^i_k\otimes u^k_j\;\forall\;i\leq n$ and $\delta(u^{n+1}_j)=1\otimes u^{n+1}_j$. The coinvariant subalgebra is $\imath\,\mc{O}(S^{2n+1}_q)$, and we can think of $S^{2n+1}_q$ as a ``quotient space''
$SU_q(n+1)/SU_q(n)$. We can write symbolically:
$$
\textup{\large ``}\; 
SU_q(n)
\lhook\joinrel\longrightarrow
SU_q(n+1)
\relbar\joinrel\twoheadrightarrow
S^{2n+1}_q
\;\textup{\large ''} \;.
$$
It was proved in \cite{Mey95} that this is a quantum principal bundle in the sense of \cite{BM92}, and then a Hopf-Galois extension \cite[Prop.~1.6]{Haj96}.
\end{ex}

\subsection{Strongly graded algebras and Morita equivalence}
In this section, I will first recall some properties of principal bundles whose Hopf algebra is a group algebra $\Bbbk G$, and then focus on the group $G=\Z$, i.e.~on actions of the dual $\widehat G=U(1)$.

\begin{df}[\cite{NvO82}]
Let $G$ be a group.
A $\Bbbk$-algebra $A$ is called \emph{$G$-graded} if $A=\bigoplus_{g\in G}A_g$ decomposes as direct sum of vector subspaces labeled by elements of $G$ such that \mbox{$A_gA_h\subset A_{gh}$} $\forall\;g,h\in G$.
It is \emph{strongly} $G$-graded if last inclusion is an equality:
$A_gA_h=A_{gh}\;\forall\;g,h\in G$.
\end{df}

If $A$ is $G$-graded, $A_gA_{g^{-1}}\subset A_e$ is a two-sided ideal in $A_e$. (Exercise: prove this statement.) The condition of strong grading is equivalent to the request that these ideals are trivial.

\begin{lemma}[\cite{NvO82}]\label{lemma:AA}
$A$ is strongly $G$-graded $\iff$ $A_gA_{g^{-1}}=A_e\;\forall\;g\in G$.
\end{lemma}

\begin{proof}
We prove ``$\Leftarrow$'', the other implication being trivial. For all $g,h\in G$:
$$
A_{gh}=A_{gh}A_e=(A_{gh}A_{h^{-1}})A_h\subset A_gA_h \subset A_{gh} \;.
$$
Here the first equality holds because $A_e$ is unital,
in the second equality I used the hypothesis $A_e=A_{h^{-1}}A_h$, and last two inclusions hold by the definition of $G$-grading. The above chain of inclusions clearly imply $A_gA_h=A_{gh}\;\forall\;g,h$, hence the thesis.
\end{proof}

Any $G$-graded algebra is canonically a $\Bbbk G$-comodule algebra with respect to the action defined on homogeneous elements by:
$$
\delta(a)=a\otimes g\;,\qquad\forall\;a\in A_g.
$$
Note that the subalgebra of coinvariants is $A^{\Bbbk G}=A_e$, where $e\in G$ is the unit element. There is a simple characterization of Hopf-Galois extension in this case:

\begin{thm}[\cite{Ulb82,Mon93}]\label{thm:1.22}
$A^{\Bbbk G}\hookrightarrow A$ is Hopf-Galois $\iff$ $A$ is strongly $G$-graded.
\end{thm}

\begin{proof}
The proof is simple and instructive. Let $H=\Bbbk G$ and $B:=A^H=A_e$.
The canonical map \eqref{eq:can}, let us denote it by $\psi$, is given by $\psi(a\otimes_Bb)=\sum_{g\in G}ab_g\otimes g$ for all $a,b\in A$ and with $b_g$ the component of $b$ in $A_g$.

\medskip

\noindent
\textit{Proof of} ``$\Leftarrow$''. If $A$ is strongly $G$-graded, from $A_{g^{-1}}A_g=A_e$ we deduce that $1\in A_e$ can be written as
\begin{equation}\label{eq:frame}
1=\sum\nolimits_{i=1}^{N_g}\xi_g^i\eta_g^i
\end{equation}
for some $\xi^i_g\in A_{g^{-1}}$ and $\eta^i_g\in A_g$, and for all $g\in G$. Let $\varphi:A\otimes H\to A\otimes_BA$ be the map defined by:
$$
\varphi(a\otimes g)=
\sum\nolimits_{i=1}^{N_g}a\xi_g^i\otimes_B\eta_g^i
\;,\qquad\forall\;a\in A,\;g\in G.
$$
From \eqref{eq:frame} it easily follows that $\psi\circ\varphi$ is the identity on $A\otimes H$ and, for all $a,b\in A$:
$$
\varphi\circ\psi(a\otimes_Bb)=\sum_{g\in G}\sum_{i=1}^{N_g}ab_g\xi_g^i\otimes_B\eta_g^i \;.
$$
Since $b_g\xi_g^i\in B$ and the tensor product is over $B$, we get
$$
\varphi\circ\psi(a\otimes_Bb)=\sum_{g\in G}\sum_{i=1}^{N_g}a
\otimes_Bb_g\xi_g^i\eta_g^i=a\otimes_Bb \;.
$$
Thus, $\varphi$ is the inverse of $\psi$, and the canonical map is bijective.

\medskip

\noindent
\textit{Proof of} ``$\Rightarrow$''. Suppose $\psi$ is invertible.
Then $\psi^{-1}(1\otimes g)=\sum\nolimits_{i=1}^{N_g}\xi_g^i\otimes_B\eta_g^i$ for some $\xi_g^i,\eta_g^i\in A$. Applying $\psi$ both sides one proves that $\xi^i_g\in A_{g^{-1}}$ and $\eta^i_g\in A_g$, and that \eqref{eq:frame} is satisfied. This implies $A_{g^{-1}}A_g=A_e\;\forall\;g$, which using Lemma \ref{lemma:AA} concludes the proof.
\end{proof}

Note that previous theorem is proved by constructing a generating family for $A_g$. For all $a\in A_g$, from \eqref{eq:frame}:
$$
a=\sum (a\xi^i_g)\eta^i_g=\sum \xi^i_{g^{-1}}(\eta^i_{g^{-1}}a) \;.
$$
Since $a\xi^i_g\in A_e$ and $\eta^i_{g^{-1}}a\in A_e$, the above identities tells us that $\{\eta^i_g\}$ is a generating family for $A_g$ as a left $A_e$-module, and $\{\xi^i_{g^{-1}}\}$ is a generating family for $A_g$ as a right $A_e$-module.
In particular, $A_g$ is finitely generated as a left/right module  as soon as $A$ is strongly graded.

The existence of a resolution of the unit as in \eqref{eq:frame} is in fact equivalent to the condition $A_{g^{-1}}A_g=A_e$ of Lemma \ref{lemma:AA}.

\medskip

Crossed products provide a family of examples of strongly graded algebras.

\begin{ex}\label{ex:crossed}
Let $\alpha:G\to\mathrm{Aut}(B)$ be an action of a group $G$ on a $\Bbbk$-algebra $B$. The \emph{crossed product} $A:=B\rtimes G$ has underlying vector space $B\otimes G$ and multiplication
$$
(a\otimes g)(b\otimes h)=a\,\alpha_g(b)\otimes gh \;,\qquad\forall\;a,b\in B,\;g,h\in G.
$$
It is strongly $G$-graded, with natural grading given by $A_g:=\mathrm{Span}\{a\otimes g:a\in B\}$.
\end{ex}

\medskip

Not all strongly graded algebras are crossed product (cf.~Example 2.9 of \cite{Mon09}).

\begin{df}[Morita equivalence]
Let $A,B$ be two $\Bbbk$-algebras. An $A$-$B$-bimodule $M$ is called a \emph{Morita equivalence bimodule} if there exists a $B$-$A$-bimodule $N$ such that $M\otimes_BN\simeq A$ as an $A$-bimodule and $N\otimes_AM\simeq B$ as a $B$-bimodule. If a Morita equivalence bimodule exists, $A$ and $B$ are called \emph{Morita equivalent} (and one can verify that it is an equivalence relations). If $A=B$ we talk about \emph{self Morita equivalence bimodule} (or SMEB).
\end{df}

\begin{ex}
For any $n\geq 1$,
$\Bbbk^n$ is a Morita equivalence bimodule between $\Bbbk$ and the matrix algebra $M_n(\Bbbk)$. More generally, for any (unital $\Bbbk$-)algebra $A$, the free module
$A^n$ is a Morita equivalence $M_n(A)$-$A$-bimodule, and any finitely generated projective right $A$-module is a Morita equivalence $B$-$A$ equivalence bimodule with $B:=\mathrm{End}_A(M)$
(this is discussed, for example, in Appendix A.2 and A.3 of \cite{Lan02}).
\end{ex}

\begin{prop}
If $A=\bigoplus_{g\in G}A_g$ is strongly $G$-graded, then $A_g$ is a SMEB for all $g$.
\end{prop}

\begin{proof}
By Lemma \eqref{lemma:AA}, $A$ is strongly graded if{}f for all $g$ there exists a resolution of the unit as in \eqref{eq:frame}.
We have to show that \eqref{eq:frame} is equivalent to the property of the $A_e$-bimodules $A_g$ of being SMEBs, i.e.~that there are bimodule isomorphisms $A_{g^{-1}}\otimes_{A_e}A_g\simeq A_e$ (for all $g$).

Similarly to the proof of Theorem \ref{thm:1.22},
it follows from \eqref{eq:frame} that the $A_e$-bimodule map
$$
\varphi:A_{g^{-1}}\otimes_{A_e}A_g\to A_e\;,\qquad
a\otimes_{A_e} b\mapsto ab \;,
$$
has inverse
$$
\varphi^{-1}:A_e\to A_{g^{-1}}\otimes_{A_e}A_g \;,\qquad
a\mapsto \sum\nolimits_ia\xi^i_g\otimes_{A_e}\eta^i_g \,
$$
as one can easily check.
\end{proof}

One could also prove the converse: if $A_g$ is a SMEB, with isomorphism $A_{g^{-1}}\otimes_{A_e}A_g\to A_e$ given by the multiplication map, then $A$ is strongly graded.

\smallskip

The connection with Morita equivalence allows for a nice geometrical interpretation of noncommutative principal bundles, which I will discuss for $G=\Z$. In this case, one can improve Lemma \ref{lemma:AA} as follows.

\begin{lemma}\label{lemma:1.27}
$A=\bigoplus_{n\in\Z}A_n$ is strongly $\Z$-graded $\iff$ $A_1A_{-1}=A_{-1}A_1=A_0$.
\end{lemma}
\begin{proof}
``$\Rightarrow$'' is trivial ($A_jA_k=A_{j+k}$ by definition).
Concerning ``$\Leftarrow$'', using $A_{-1}A_1=A_0$ and proceeding as in the proof of Lemma \ref{lemma:AA}, we deduce
$$
A_k=A_kA_0=A_kA_{-1}A_1\subset A_{k-1}A_1\subset A_k \;,
$$
that is $A_{k-1}A_1=A_k$. Similarly one proves that
$A_1A_{k-1}$, $A_{k+1}A_{-1}$ and $A_{-1}A_{k+1}$ are all equal to $A_k$.
With these four relations, by induction we prove $A_hA_k=A_{h+k}\;\forall\;h,k\in\Z$ (the details are left as an exercise; hint: one has to distinguish four cases, and in each one use one of the four relations mentioned above).
\end{proof}

On the other hand, it is possible to reconstruct a strongly $\Z$-graded algebra from an algebra and two SMEBs.
Given the data of an algebra $B_0$ and two $B_0$-bimodules, we set (for $n\geq 1$):
$$
B_n=\underbrace{B_1\otimes_{B_0}B_1\otimes_{B_0}\ldots\otimes_{B_0} B_1}_{n\text{ times}} \;,\qquad
B_{-n}=\underbrace{B_{-1}\otimes_{B_0}B_{-1}\otimes_{B_0}\ldots\otimes_{B_0} B_{-1}}_{n\text{ times}} \;.
$$
Then: both $B_+:=\bigoplus_{n>0}B_n$ and $B_-:=\bigoplus_{n<0}B_n$
are graded associative (non-unital) algebras w.r.t.~the tensor product;
$B_0\oplus B_+$ and $B_0\oplus B_-$ are graded associative unital algebras if the product by $B_0$ is given by the bimodule structure.
Since $B_+$ and $B_-$ are generated respectively by $B_1$ and $B_{-1}$, in order to extend the product to $B:=\bigoplus_{n\in\Z}B_n$ we need two bimodule maps
$$
\varphi:B_1\otimes_{B_0}B_{-1}\to B_0 \;,\qquad
\psi:   B_{-1}\otimes_{B_0}B_1\to B_0 \;,
$$
that are used to define $\xi\cdot\eta:=\varphi(\xi,\eta)$ and $\eta\cdot\xi:=\psi(\eta,\xi)$ for $\xi\in B_1$ and $\eta\in B_{-1}$.
Associativity requires
\begin{equation}\label{eq:phipsi}
\varphi(a,b)\hspace{1pt}c=a\hspace{1pt}\psi(b,c)
\qquad\text{and}\qquad
b\,\varphi(c,d)=\psi(b,c)d
\end{equation}
for all $a,c\in B_1$ and $b,d\in B_{-1}$.
In this way one gets a $\Z$-graded algebra $B$.
It is not difficult to prove that it is strongly graded if{}f $\varphi,\psi$ are bijective, that is if{}f $B_1,B_{-1}$ are SMEBs.

\begin{es}
Prove that the algebra $B$ above is strongly graded if and only if
$\varphi,\psi$ are bijective. Hint: the only non-trivial part is to prove that strongly graded implies injectivity, which can be done using \eqref{eq:phipsi}.
\end{es}

The construction above has a nice geometrical interpretation: a strongly $\Z$-graded algebra is the algebraic counterpart of a principal $U(1)$-bundle, and reconstructing it from two SMEBs is equivalent to reconstructing the principal bundle from a pair of dual line bundles. This is briefly explained in the next section.

\subsection{Pimsner's algebras and (noncommutative) principal \texorpdfstring{$U(1)$}{U(1)}-bundles}\label{sec:1.5}
The construction of previous section of a strongly $\Z$-graded algebra from self Morita equivalence bimodules can be better understood in the $C^*$-algebraic setting. 
Central in the construction is the notion of \emph{strong} Morita equivalence, introduced by M.A.~Rieffel \cite{Rie74}.

Suppose we have a $C^*$-algebra $A_0$ and a pair $(E,\phi)$ of a right Hilbert $A_0$-module and an injective $C^*$-homomorphism
$\phi$ from $A_0$ to the algebra of adjointable endomorphisms of $E$.
With this data one can construct a $C^*$-algebra $A$ containing $A_0$, called the Pimsner algebra of $E$ \cite{Pim97}. When $E$ is a SMEB, $A$ contains a dense strongly $\Z$-graded algebra, whose degree zero part is dense in $A_0$. In a purely algebraic setting, this is essentially the construction discussed at the end of previous section. In fact, such a construction generalize both crossed product by $\Z$, as well as other important classes of $C^*$-algebras, like Cunts-Krieger algebras.
I will not discuss the details here: the interested reader can see e.g.~\cite{Ari15,ABL14,AKL14,ADL15}.

A celebrated example of noncommutative space, that is described by a Pimsner algebra is the noncommutative torus (which is, in fact, a crossed product $C(S^1)\rtimes\Z$).
Quantum lens spaces (see next section) are examples of Pimsner algebras that are not crossed product.

\smallskip

The relation with principal $U(1)$-bundles is the following. Let $A=C(X)$ be the $C^*$-algebra of continuous function on a compact topological Hausdorff space (note that any commutative unital $C^*$-algebra has this form, by Gelfand-Naimark theorem).
The module $M$ of continuous sections of a Hermitian vector bundle $V\to X$ (with finite-dimensional fiber) is an example of strong Morita equivalence bimodule between $A$ and $B=\mathrm{End}_A(M)$ \cite[Ex.~2]{Rie82}, and by Serre-Swan theorem every finitely generated projective $A$-module is of this form. On the other hand, there are examples (not finitely generated) that do not come from vector bundles: for $\C=C(\{\mathrm{pt}\})$ the algebra of functions on one point, $\ell^2(\N)$ is a strong Morita equivalence bimodule between $\C$ and the algebra of compact operators $\mc{K}(\ell^2(\N))$.

On the other hand, in the unital case a SMEB is always finitely generated projective. Indeed, if $A$ and $B$ are both unital $C^*$-algebras (not necessarily commutative), every strong Morita equivalence $A$-$B$-bimodule is finitely generated and projective (see \cite[Ex.~4.20]{GVF01} or \cite[p.~291]{Rie82}). In particular, for $A=C(X)$ every SMEB $M$ is 
the module of sections of a Hermitian vector bundle $V\to X$, and since $\mathrm{End}_A(M)=A$ we deduce that $V$ is a line bundle.

\smallskip

Let now $P\to X$ be a principal $U(1)$-bundle. Then, there is a right coaction $\delta$ of the (commutative) Woronowicz $C^*$-algebra $Q:=C(U(1))$ on $A:=C(P)$ given by the pullback of the action of $U(1)$ on $P$, and the subalgebra of coinvariants can be identified with $A_0:=C(X)$. Let $u$ be the (unitary) generator of $Q$, and $A_n:=\{a\in A:\delta(a)=a\otimes u^n\}$ the corresponding weight spaces, $n\in\Z$. Then, $\bigoplus_{n\in\Z}A_n$ is dense in $A$; it is strongly $\Z$-graded, the action being principal, and each $A_n$ is a SMEB for $A_0$, hence the module of sections of a Hermitian line bundle on $X$. From a geometric point of view, Pimsner's construction allows to reconstruct the total space $P$ of a principal $U(1)$-bundle on $X$ from a line bundle on the base space.
Note that there is a isomorphism of vector spaces
$A_1\to A_{-1}$, $a\mapsto a^*$, and this is the reason why we need just one line bundle to reconstruct $A$.

\begin{rem}
Let $P\to X$ be a principal $U(1)$-bundle. Then $C(P)$ is an example of Pimsner algebra that is not a crossed product by $\Z$, unless $P\simeq X\times U(1)$ is a trivial bundle. Indeed, both $C(P)$ and $C(X)$ are commutative, so $C(P)=C(X)\rtimes\Z$ implies that the action of $\Z$ is trivial. Therefore $C(P)\simeq C(X)\botimes C_r^*(\Z)\simeq C(X)\botimes C(U(1))$, which implies $P\simeq X\times U(1)$.
\end{rem}

\subsection{Quantum weighted projective spaces}
This section is devoted to the example of quantum lens spaces $L_q^n(p;\mv{\ell})$ and quantum weighted projective spaces $\WP^n_q(\mv{\ell})$, from \cite{DL14}. The first two arrows of the following diagram have been already illustrated in Example \ref{ex:SUqn} and \ref{ex:Sqn}. I will now focus on the framed part of the diagram.

\smallskip

\begin{center}
\begin{tikzpicture}
\tikzset{
    iniet/.style={right hook->, shorten >=1pt, shorten <=2pt},
		suriet/.style={->, shorten >=1pt, shorten <=1pt},
}

\node (A) {$SU_q(n)$};
\node[right=of A] (B) {$SU_q(n+1)$};
\node[below=of B] (C) {$S^{2n+1}_q$};
\node[right=of C] (E) {$L_q^n(p;\mv{\ell})$};
\node[below=of E] (G) {$\WP^n_q(\mv{\ell})$};

\draw[iniet] (A) -- (B);
\draw[suriet] (B) -- (C);
\draw[suriet] (C) -- node[above]{\small $\Z_p$} (E);
\draw[suriet] (E) -- node[right]{\small $U(1)$} (G);
\draw[suriet] (C) -- node[below left]{\small $U(1)$} (G);

\node [above=0pt of E] (L) {\rule{50pt}{0pt}};
\node [below=2.3cm of L] (M) {\rule{50pt}{0pt}};
\node [left=1.9cm of L] (N) {};
\node [left=1.9cm of M] (O) {};

\draw[color=gray,thick,dashed] (O.south west) -- (N.north west) -- (L.north east) -- (M.south east) -- cycle;

\end{tikzpicture}
\end{center}

\subsubsection{``Classical'' weighted projective spaces}
Let us start with some standard material.

\begin{df}
A \emph{weight vector} $\mv{\ell}=(\ell_0,\ldots,\ell_n)$ is a finite sequence of positive integers, called \emph{weights}.
A weight vector is:
\begin{itemize}\itemsep=1pt
\item \emph{coprime} if $\gcd(\ell_0,\ldots,\ell_n)=1$;
\item \emph{pairwise coprime} if $\gcd(\ell_i,\ell_j)=1\;\forall\;i\neq j$.
\end{itemize}
\end{df}

\noindent
Given a weight vector $\mv{\ell}=(\ell_0,\ldots,\ell_n)$, an action $U(1)\curvearrowright S^{2n+1}$
is given by:
$$
z\mapsto \big(t^{\ell_0}z_0,\ldots,t^{\ell_n}z_n\big)
$$
for all $t\in U(1)$ and $z=(z_0,\ldots,z_n)\in S^{2n+1}$ (a unit vector in $\C^{n+1}$).
The quotient (w.r.t.~the action above):
$$
\WP^n(\mv{\ell})= S^{2n+1}/U(1)
$$
is a complex projective variety called \emph{weighted (complex) projective space}.
These spaces have been classified \cite{BFNR13}, and it is well known that two of them are isomorphic as projective varieties if{}f they are homeomorphic.
In particular, $$\WP^n(1,\ldots,1)\simeq\C\mathrm{P}^n\simeq SU(n+1)/U(n)$$ is an ordinary projective space, and
$\WP^1(\ell_0,\ell_1)\simeq\C\mathrm{P}^1\;\forall\;\ell_0,\ell_1$.

As quotient spaces they have a natural orbifold structure and include, in complex dimension 1, the orbifolds named \emph{teardrops} by Thurston in \cite{Thu80} (all homeomorphic to $S^2$).

A result in \cite{DL14}, which to the best of our knowledge was not known in the literature, is the next Prop.~\ref{prop:classical}.

\begin{df}[The sharp map]
Given a weight vector $\mv{\ell}=(\ell_0,\ldots,\ell_n)$,
we denote by $\mv{\ell}^\sharp$ the weight vector with $i$-th component equal to $\prod_{j\neq i}\ell_j$.
\end{df}

\begin{ex}
$(\ell_0,\ell_1)^\sharp=(\ell_1,\ell_0)$ and
$(\ell_0,\ell_1,\ell_2)^\sharp=(\ell_1\ell_2,\ell_0\ell_2,\ell_0\ell_1)$
for all $\ell_0,\ell_1,\ell_2$.
\end{ex}

\medskip

\noindent
The map $\sharp:\mv{\ell}\mapsto\mv{\ell}^\sharp$ is almost an involution. For all weight vectors:
$$
(\mv{\ell}^\sharp)^\sharp=k\mv{\ell}
$$
with $k=(\ell_0\ell_1\ldots\ell_n)^{n-1}$.
Since trivially $\WP^n(\mv{\ell})\simeq\WP^n(k\mv{\ell})$ for all $k\geq 1$,
we don't loose generality if we assume one of the following properties: \vspace{3pt}
\begin{center}
i) $\;\mv{\ell}$ is coprime,
\qquad\qquad or \qquad\qquad
ii) $\;\mv{\ell}\in\mathrm{Im}(\sharp)$. \vspace{3pt}
\end{center}
While we can always assume either (i) or (ii), in general we cannot assume that both are satisfied. A result of \cite{DL14} is, in fact, that (i) and (ii) are satisfied simultaneously if{}f $\WP^n(\mv{\ell})\simeq\C\mathrm{P}^n$.

\begin{prop}~\label{prop:classical}
\begin{list}{$\bullet$}{\leftmargin=1.5em \itemsep=3pt}

\item $\mv{\ell}^\sharp$ is coprime $\quad\iff\quad\mv{\ell}$ is pairwise coprime.

\item $\WP^n(\mv{\ell}^\sharp)\simeq\C\mathrm{P}^n\!\quad\iff\quad\mv{\ell}$ is
pairwise coprime.

\item Equivalently: let $\mv{\ell}$ be coprime; then:
$\quad
\WP^n(\mv{\ell})\simeq\C\mathrm{P}^n\quad\iff\quad\mv{\ell}\in\mathrm{Im}(\sharp) \;.
$
\end{list}
\end{prop}

\subsubsection{Quantum lens spaces and weighted projective spaces}
The definition of quantum weighted projective spaces (``QWP spaces'', from now on) and lens spaces is straightforward. One simply replaces $S^{2n+1}$ by $S^{2n+1}_q$ (Example \ref{ex:Sqn}) and proceeds as in the classical case.

Let $\mv{\ell}$ be a weight vector.
An action $\alpha:U(1)\to\mathrm{Aut}\,\mc{O}( S^{2n+1}_q)$ is defined on generators by:
$$
\alpha_t(z_i):=t^{\ell_i}z_i \;\forall\;i=0,\ldots,n,\;t\in U(1).
$$
Let $p\geq 1$ and $\Z_p\subset U(1)$ the subgroup of $p$-th roots of unity. 
We define the algebras of QWP and lens spaces as fixed point subalgebras for the action of $U(1)$ and $\Z_p$ respectively:
$$
\mc{O}(L_q^n(p;\mv{\ell})):=\mc{O}( S^{2n+1}_q)^{\Z_p}
\;,\quad\qquad
\mc{O}(\WP^n_q(\mv{\ell})):=\mc{O}( S^{2n+1}_q)^{U(1)}
\;.
$$
As a special case,
$\WP^n_q(1,\ldots,1)=\C\mathrm{P}_q^n$
are the standard quantum projective spaces studied for example in \cite{DD09,DL10}.

Next theorem appeared first in \cite{DL14} for a special class of weights; then it was observed in \cite{BF15} that the same proof applies to a more general case, cf.~point (ii) of the theorem below. The statement in point (i) is also more general than the one appearing in \cite{DL14}.

\begin{thm}[\cite{DL14,BF15}]
Let $\mv{\ell}$ be any weight, and
set $p:=\ell_0\ell_1\cdots\ell_n$. The following extensions are Hopf-Galois:\vspace{5pt}
\begin{center}
\begin{tabular}{cp{7mm}c}
(i)\quad
$\mc{O}(\WP^n_q(\mv{\ell}^\sharp))\hookrightarrow\mc{O}(L_q^n(p;\mv{\ell}^\sharp))$ &&
(ii)\quad
$\mc{O}(\WP^n_q(\mv{\ell}))\hookrightarrow
\mc{O}(L_q^n(p;\mv{\ell}))$ \\[3pt]
{\scriptsize[FD \& G.~Landi, 2015]} &&
{\scriptsize[T. Brzezi{\'n}ski \& S.A. Fairfax, 2015]}
\end{tabular}
\end{center}
\end{thm}

\noindent
\textit{Sketch of the proof.}
Let $A:=\mc{O}(L_q^n(p;\mv{\ell}^\sharp))$ in case (i) and $A:=\mc{O}(L_q^n(p;\mv{\ell}))$ in case (ii).
Call:
$$
A_n:=\big\{a\in A:\alpha_t(a)=t^{pn}a\;\forall\;t\in U(1)\big\}
$$
(these are the only non-zero weight spaces of $U(1)$ in $A$),
and note that
$A_0\equiv\mc{O}(\WP^n_q(\mv{\ell}^\sharp))$ in case (i) and $A_0\equiv\mc{O}(\WP^n_q(\mv{\ell}))$ in case (ii).
We want to prove that $A$ is strongly $\Z$-graded, which by Lemma \ref{lemma:1.27} means proving that $A_1A_{-1}$ and $A_{-1}A_1$ contain the unit element of $A$.
For the proof we need a basic result from algebraic geometry:

\begin{lemma}[Hilbert's weak Nullstellensatz]\label{lemma:Null}
An ideal $I\subset R:=\C[x_1,\ldots,x_n]$ contains $1$ if and only if the polynomials in $I$ do not have common zeros, 
i.e. $\mathcal{Z}(I)=\emptyset$. (In other words, the only ideal representing the empty variety is $R$.)
\end{lemma}

With this, one proves the theorem in three steps:
\begin{enumerate}\itemsep=2pt
\item
Check that the elements $x_i:=z_iz_i^*$ belong to a commutative subalgebra $R\simeq\C[x_1,\ldots,x_n]$ of $A_0$ (in particular $x_0=1-x_1-x_2\ldots-x_n\in R$).

\item
Prove by induction that $z_{\smash{i}}^{\ell_i}(z_i^*)^{\ell_i}=P_i(x_1,\ldots,x_n)$ are polynomials with no common zeroes.

\item
The ideal $\langle P_0,\ldots,P_n\rangle\subset R$ is then trivial (Lemma \ref{lemma:Null}), i.e.~$\exists\;a_0,\ldots,a_n\in R$ s.t.
$$
\sum\nolimits_ia_iP_i=
\sum\nolimits_ia_iz_{\smash{i}}^{\ell_i}
(z_i^*)^{\ell_i}=
1 \;.
$$
Since $a_iz_{\smash{i}}^{\ell_i}\in A_1$ and
$(z_i^*)^{\ell_i}\in A_{-1}$,
this proves $A_1A_{-1}=A_0\,$. Similarly one proves $A_{-1}A_1=A_0$. \qed
\end{enumerate}

\begin{es}
Using the defining relations in Example \ref{ex:Sqn}, prove that
$$
P_i:=z_{\smash{i}}^{\ell_i}(z_i^*)^{\ell_i}=
\prod_{k=0}^{\ell_i-1}\Big\{x_i+(1-q^{-2k})\sum\nolimits_{j>i}x_j\Big\}
$$
for all $i=0,\ldots,n$.
With this, prove by induction on $k$ (from $n$ to $0$) that $P_k,\ldots,P_n$ have no common zeros.
\end{es}


\section{Quasi-associative algebras and cochain quantization}

In the first part of these lectures I recalled the definitions of
Hopf-algebra, algebra module, etc. over a ground field $\Bbbk$. These definitions remain valid when the ground field is replaced by a commutative (unital) ring $R$. In this section, we will be mainly interested in the ring $R=\C\lbrak\hbar\rbrak$ of formal power series in $\hbar$.
In this case, by a Hopf algebra over $\C\lbrak\hbar\rbrak$ I shall always mean a \emph{topological} Hopf algebra, completed in the h-adic topology, and by $\otimes$ the completed tensor product over $\C\lbrak\hbar\rbrak$ (see e.g.~\cite[\S4.1A]{CP94}).

\subsection{Group cohomology and twists}
Let us start by recalling the definition of twisted group algebra.
For the analogous construction in the framework of $C^*$-algebras one can see \cite[\S III.1.7]{Lan98}, where the case of unimodular Lie groups is discussed, and
\cite[App.~D.3]{Wil07} for general locally compact groups.

\begin{df}
Let $G$ be a group and $F:G\times G\to\Bbbk$ a map. The \emph{twisted group algebra} $\Bbbk_FG$ has underlying vector space isomorphic to $\Bbbk G$ and product defined on basis elements by
\begin{equation}\label{eq:twist}
g\ast h=F(g,h) g\cdot h \;,
\end{equation}
where the one on the right hand side is the product of $G$.
\end{df}

One can easily work out the necessary and sufficient condition for which \eqref{eq:twist} is associative, obtaining the definition of $2$-cocycle in the group cohomology of $G$. The product \eqref{eq:twist} is associative if{}f
\begin{equation}\label{eq:2cocycle}
F(a,b)F(ab,c)=F(a,bc)F(b,c) \;,\qquad\forall\;a,b,c\in G,
\end{equation}
and $1\in\Bbbk G$ is the neutral element of the twisted product if{}f
\begin{equation}\label{eq:unital}
F(1,g)=F(g,1)=1 \;,\qquad\forall\;g\in G.
\end{equation}
A function $F:G\times G\to\Bbbk\smallsetminus\{0\}$ satisfying \eqref{eq:2cocycle} is called a \emph{group $2$-cocycle}, since this condition can be interpreted as the vanishing of the coboundary $\partial F$ of $F$ in the group cohomology of $G$, given by:
\begin{equation}\label{eq:3cocycle}
\partial F(a,b,c)=\frac{F(a,bc)F(b,c)}{F(a,b)F(ab,c)} \;.
\end{equation}

More generally, if $C^n$ is the (Abelian)
group of functions $\overbrace{G\times\ldots\times G}^{n\text{ times}}\to\Bbbk\smallsetminus\{0\}$, with pointwise product, a coboundary operator $C^n\to C^{n+1}$ is given by \cite[\S2.3]{Maj95}:
$$
\partial\varphi(g_1,\ldots,g_{n+1})=
\prod_{i=0}^{n+1}\varphi(g_1,\ldots,g_i\cdot g_{i+1},\ldots,g_{n+1})^{(-1)^i} \;,
$$
where, by convention, the $i=0$ factor is $\varphi(g_2,\ldots,g_{n+1})$ and the $i=n+1$ factor is $\varphi(g_1,\ldots,g_n)$.

\begin{df}
A function $F$ satisfying \eqref{eq:2cocycle} and \eqref{eq:unital} is called a \emph{unital} $2$-cocycle, or a \emph{multiplier}.
The algebra $\Bbbk_FG$ will be called a \emph{cocycle quantization} or \emph{cocycle twist} of $\Bbbk G$.
\end{df}

\begin{ex}[Noncommutative torus]\label{ex:2.3}
Let $G=\Z^2$. Via Fourier transform, we can think of the algebra $\Bbbk G$ with convolution product as a dense subalgebra of $C(\mathbb{T}^2)$. A linear basis is given by $\{U^mV^n\}_{m,n\in\Z}$, with $U,V$ the generators of $\Z^2$. For $\theta\in\R$,
a unital $2$-cocycle is given by
$$
F(U^jV^k,U^mV^n):=e^{\pi i\theta (jn-km)} \;.
$$
for all $j,k,m,n\in\Z$.
The algebra $\Bbbk_FG$ is generated by $U,V$ with relation
$$
U*V=e^{2\pi i\theta}V*U \;,
$$
and it is a $*$-algebra with involution $U^*=U^{-1}$ and $V^*=V^{-1}$. The $C^*$-algebra $C(\T^2_\theta)$ of the noncommutative torus, see e.g.~\cite[\S4]{Var06}, can be defined as the closure of $\Bbbk_FG$ in the operator norm coming from the representation on $L^2(S^1)$ given by $Uf(z):=f(e^{2\pi i\theta}z)$, $Vf(z):=zf(z)$ (hence $C(\T^2_\theta)$ can be defined as a crossed product $C(S^1)\rtimes\Z$).
\end{ex}

\smallskip

In general, when \eqref{eq:2cocycle} is not satisfied, one gets a non-associative algebra, whose non-associativity is ``controlled'' by the $3$-cocycle $\Phi=\partial F$. That is
$$
a*(b*c)=\Phi(a,b,c)\,(a*b)*c \;,\qquad\forall\;a,b,c\in G.
$$
There are interesting examples where \eqref{eq:2cocycle} is not satisfied, a celebrated one is the octonion example of Albuquerque and Majid \cite{AM99,AM98}.
I will review this example in the next section.

I will then recall -- in \S\ref{sec:Hopf} -- the generalization of the theory of cocycle and cochain quantization to Hopf algebras. Here the main difference is that one has a sequence of non-Abelian groups, and the map $\partial$ is a group homomorphism and a coboundary operator only in the Abelian case (that is, for commutative Hopf algebras). It is possible, nevertheless, to give a cohomological interpretation to deformation theory via Hopf algebras in the framework of \emph{non-Abelian cohomology} \cite{Gir71}. This is recalled in \S\ref{sec:nonAb}.

Then, I will present some results from \cite{DF15} about the noncommutative torus, namely how to derive its modules from a non-associative deformation of a Heisenberg manifold.

\subsection{Octonions}
(From \cite{AM99,AM98})
Let us use a multiplicative notation for the 
group $\Z_2$, hence $\Z_2=\{+1,-1\}$.
Consider the group algebra $\R[\Z_2^3]$. As an abstract algebra, this has a linear
basis given by elements $1,e_1,\ldots,e_7$ and the multiplication is uniquely determined by
the rules:
\begin{itemize}
\item $1$ is the unit element;
\item $a^2=1$ for any point of the Fano plane (Fig.~\ref{fig1});
\item for any two distinct points $a,b$ of the Fano plane, the product $ab=c$ is given by the
(unique) element $c$ lying on the same line or circle with $a$ and $b$.
\end{itemize}
An isomorphism between this abstract algebra and $\R[\Z_2^3]$ is given by
\begin{align*}
e_0:=1 &\mapsto (1,1,1) \;,&
e_1 &\mapsto (-1,1,1) \;,&
e_2 &\mapsto (1,-1,1) \;,&
e_3 &\mapsto (-1,-1,1) \;,\\
e_4 &\mapsto (-1,-1,-1) \;,&
e_5 &\mapsto (1,-1,-1) \;,&
e_6 &\mapsto (-1,1,-1) \;,&
e_7 &\mapsto (1,1,-1) \;.
\end{align*}

On the other hand the octonion algebra $\mathbb{O}$ has vector space basis $1,e_1,\ldots,e_7$ over $\R$ , and
the product is uniquely determined by the rules:
\begin{itemize}
\item $1$ is the unit element;
\item $a^2=-1$ for any vertex of the oriented graph in Fig.~\ref{fig2};
\item $ab=-ba=c$ for any ordered triple $(a,b,c)$ of vertices lying on the same line or circle of the graph,
and for any cyclic permutation of such a triple.
\end{itemize}

\begin{figure}[t]
  \centering
  \subfloat[Fano plane]{
  \begin{tikzpicture}

\def \radius {1.5cm}
\def \margin {15}

\draw[gray!60,thick] ({30+\margin}:\radius) 
    arc ({30+\margin}:{150-\margin}:\radius);
\node[draw, circle] at (30:\radius) {$e_2$};

\draw[gray!60,thick] ({150+\margin}:\radius) 
    arc ({150+\margin}:{270-\margin}:\radius);
\node[draw, circle] at (150:\radius) {$e_1$};

\draw[gray!60,thick] ({270+\margin}:\radius) 
    arc ({270+\margin}:{390-\margin}:\radius);
\node[draw, circle] at (270:\radius) {$e_3$};

\coordinate (e1) at (150:\radius);
\coordinate (e2) at (30:\radius);
\coordinate (e3) at (270:\radius);
\coordinate (e4) at (0,0);
\coordinate (e5) at (330:2*\radius);
\coordinate (e6) at (210:2*\radius);
\coordinate (e7) at (90:2*\radius);

\node[draw, circle] at (e4) {$e_4$};
\node[draw, circle] at (e5) {$e_5$};
\node[draw, circle] at (e6) {$e_6$};
\node[draw, circle] at (e7) {$e_7$};

\draw[thick,style={shorten <=11pt,shorten >=11pt}] (e6)--(e1);
\draw[thick,style={shorten <=11pt,shorten >=11pt}] (e1)--(e7);
\draw[thick,style={shorten <=11pt,shorten >=11pt}] (e7)--(e2);
\draw[thick,style={shorten <=11pt,shorten >=11pt}] (e2)--(e5);
\draw[thick,style={shorten <=11pt,shorten >=11pt}] (e5)--(e3);
\draw[thick,style={shorten <=11pt,shorten >=11pt}] (e3)--(e6);
\draw[thick,style={shorten <=11pt,shorten >=11pt}] (e2)--(e4);
\draw[thick,style={shorten <=11pt,shorten >=11pt}] (e4)--(e6);
\draw[thick,style={shorten <=11pt,shorten >=11pt}] (e3)--(e4);
\draw[thick,style={shorten <=11pt,shorten >=11pt}] (e4)--(e7);
\draw[thick,style={shorten <=11pt,shorten >=11pt}] (e1)--(e4);
\draw[thick,style={shorten <=11pt,shorten >=11pt}] (e4)--(e5);

\end{tikzpicture}
  \label{fig1} }
	\hspace{1.2cm}
  \subfloat[Octonions multiplication]{
\begin{tikzpicture}

\tikzset{>=stealth}

\def \radius {1.5cm}
\def \margin {15}

\node[draw, circle] at (30:\radius) {$e_2$};
\draw[gray!60,thick,<-] ({30+\margin}:\radius) 
    arc ({30+\margin}:{150-\margin}:\radius);

\node[draw, circle] at (150:\radius) {$e_1$};
\draw[gray!60,thick,<-] ({150+\margin}:\radius) 
    arc ({150+\margin}:{270-\margin}:\radius);

\node[draw, circle] at (270:\radius) {$e_3$};
\draw[gray!60,thick,<-] ({270+\margin}:\radius) 
    arc ({270+\margin}:{390-\margin}:\radius);

\coordinate (e1) at (150:\radius);
\coordinate (e2) at (30:\radius);
\coordinate (e3) at (270:\radius);
\coordinate (e4) at (0,0);
\coordinate (e5) at (330:2*\radius);
\coordinate (e6) at (210:2*\radius);
\coordinate (e7) at (90:2*\radius);

\node[draw, circle] at (e4) {$e_4$};
\node[draw, circle] at (e5) {$e_5$};
\node[draw, circle] at (e6) {$e_6$};
\node[draw, circle] at (e7) {$e_7$};

\draw[thick,->,style={shorten <=11pt,shorten >=11pt}] (e6)--(e1);
\draw[thick,->,style={shorten <=11pt,shorten >=11pt}] (e1)--(e7);
\draw[thick,->,style={shorten <=11pt,shorten >=11pt}] (e7)--(e2);
\draw[thick,->,style={shorten <=11pt,shorten >=11pt}] (e2)--(e5);
\draw[thick,->,style={shorten <=11pt,shorten >=11pt}] (e5)--(e3);
\draw[thick,->,style={shorten <=11pt,shorten >=11pt}] (e3)--(e6);
\draw[thick,->,style={shorten <=11pt,shorten >=11pt}] (e2)--(e4);
\draw[thick,->,style={shorten <=11pt,shorten >=11pt}] (e4)--(e6);
\draw[thick,->,style={shorten <=11pt,shorten >=11pt}] (e3)--(e4);
\draw[thick,->,style={shorten <=11pt,shorten >=11pt}] (e4)--(e7);
\draw[thick,->,style={shorten <=11pt,shorten >=11pt}] (e1)--(e4);
\draw[thick,->,style={shorten <=11pt,shorten >=11pt}] (e4)--(e5);

\end{tikzpicture}
  \label{fig2} }
\vspace{-5pt}\caption{}
\end{figure}

\noindent
Let us identify the vector spaces underlying the algebras $\R[\Z_2^3]$ and $\mathbb{O}$, denote by $\cdot$ the product
of the former algebra and by $\ast$ the product of the latter. It is clear from the mnemonic rules above that,
for all $i,j=0,\ldots,7$:
\begin{equation}\label{eq:twistedgroupalgebra}
e_i\ast e_j=F(e_i,e_j)e_i\cdot e_j \;,
\end{equation}
where $F(e_i,e_j)\in\{\pm 1\}$ is a sign. 
The non-associativity of $\mathbb{O}$ implies that $F$ is not a $2$-cocycle.

Let $H=\R^{\Z_2^3}$ be the Hopf algebra dual to $\R[\Z_2^3]$. We can interpret $F$ as a twisting element in $H\otimes H$, cf.~\S\ref{sec:Hopf}. Since $H$ is commutative (and cocommutative) its deformation with $F$ coincide with $H$ itself, and $\mathbb{O}\simeq \R_F[\Z_2^3]$ is a $H$-module algebra too.
In fact, one can interpret $\mathbb{O}$ as a monoid in the category of representations of the Hopf algebra $H$, cf.~\cite{AM99,AM98,BM05,BM10}.

\subsection{Hopf cohomology}\label{sec:Hopf}
This section and \S\ref{sec:cochain} are devoted to
Drinfeld twists and deformations of Hopf algebras/module algebras \cite{Dri90,Dri91}.
Most of this material can be found in \cite{Maj95}.

When we talk about chain complexes, we usually think of a sequence of $\Z$-modules, i.e.~Abelian groups. Things become more complicated when non-Abelian groups are involved.

\begin{df}\label{def:chain}
A cochain complex of (possibly non-Abelian) groups $(\mathbb{G}_\bullet,\de_\bullet)$ is given by
a sequence of groups
$$
\mathbb{G}_1\xrightarrow{\;\de_1\;}\mathbb{G}_2\xrightarrow{\;\de_2\;}\ldots\xrightarrow{\de_{n-1}}\mathbb{G}_n\xrightarrow{\;\de_n\;}\ldots
$$
where each $\de_i$ is a group homomorphism and $\de_{i+1}\circ \de_i$ is the trivial homomorphism.
Using a multiplicative notation:
$\de_\bullet$ is a coboundary operator if $\de_{i+1}\circ \de_i=1$, i.e.~$\de_{i+1}\big( \de_i(a)\big)=1\;\forall\;a\in\mathbb{G}_i$.
Let $Z_i=\{a\in\mathbb{G}_i:\de_i(a)=1\}$ and $B_i=\mathrm{Im}(\de_{i-1})$. The set
$Z_i/B_i$ of all left cosets is a group (the $i$-th cohomology group)
if and only if $B_i$ is a normal subgroup of $Z_i$, which is trivially true in
the Abelian case.
\end{df}

Let $H$ be a Hopf algebra, with coproduct $\Delta$ and counit $\epsilon$ (in fact, it would be enough
for $H$ to be a bialgebra).
For $i=1,\ldots,n$, let $\epsilon_i:H^{\otimes n}\to H^{\otimes n}$ and $\Delta_i:H^{\otimes n}\to H^{\otimes n+1}$ be the maps
given by the counit and coproduct, respectively, acting on the $i$-th leg of a tensor; let $\Delta_0(h) :=1\otimes h$
and $\Delta_{n+1}(h):=h\otimes 1$ for all $h\in H^{\otimes n}$.

Let $\mathbb{G}_n$ be the multiplicative group of invertible elements of $H^{\otimes n}$
(they are Abelian groups if $H$ is commutative) and $\partial:\mathbb{G}_n\to \mathbb{G}_{n+1}$
the map given by
\begin{equation}\label{eq:mappartial}
\partial h=(\partial_+h)(\partial_-h^{-1})
\end{equation}
where
\begin{align*}
\partial_+h &:=
\prod_{\substack{0\leq i\leq n+1 \\[2pt] i\text{ even}}}\Delta_i(h)=\Delta_0(h)\Delta_2(h)\ldots
\\[3pt]
\partial_-h &:=
\prod_{\substack{0\leq i\leq n+1 \\[2pt] i\text{ odd}}}\Delta_i(h)=\Delta_1(h)\Delta_3(h)\ldots
\end{align*}
(The products is in increasing order from the left to the right.)
Recall that we use a multiplicative notation: the unit element of $\mathbb{G}_n$ is $1=1^{\otimes n}$, so an $n$-cochain $h$ is a $n$-cocycle if $\partial h=1^{\otimes n+1}$.
An $n$-cochain $h$ is called: \emph{counital} if $\epsilon_i(h)=1\;\forall\;i=1,\ldots,n$;
\emph{invariant} if it commutes with all the elements in the range of the iterated coproduct.

\smallskip

It is useful to introduce a compact notation.
For all $n\geq m$ and $1\leq i_1<i_2<\ldots<i_m\leq n$, I will denote by $h\mapsto h_{i_1\ldots i_m}$ the
linear map $H^{\otimes m}\to H^{\otimes n}$ defined on homogeneous tensors
$h=a_1\otimes a_2 \otimes\ldots\otimes a_m$ as follows: we put $a_k$ in the leg
$i_k$ of the tensor product for all $k=1,\ldots m$, and fill the additional
$n-m$ legs with $1$. The subscript $(i_ki_{k+1})$ in brackets means
that we apply the coproduct to the $k$-th factor, and put the first leg of the
coproduct in position $i_k$ and the second leg in position $i_{k+1}$.

So for example, if $h=a\otimes b$, $h_{13}=a\otimes 1\otimes b$,
$h_{(12)3}=\Delta(a)\otimes b$, $h_{1(234)}=a\otimes(\Delta\otimes\id)\Delta(b)$
(the notation for the iterated coproduct is self-explanatory, and justified by its
coassociativity).
With this notation:
\begin{subequations}
\begin{align}
\partial h &=h_1h_2(h^{-1})_{(12)} \;, && \forall\;h\in \mathbb{G}_1, \label{eq:DEsubA} \\[3pt]
\partial h &=h_{23}h_{1(23)}(h^{-1})_{(12)3}(h^{-1})_{12} \;, && \forall\;h\in \mathbb{G}_2, \label{eq:DEsubB} \\[3pt]
\partial h &=h_{234}h_{1(23)4}h_{123}(h^{-1})_{(12)34}(h^{-1})_{12(34)} \;. && \forall\;h\in \mathbb{G}_3, \label{eq:DEsubC}
\end{align}
\end{subequations}
which stands for:
\begin{align*}
\partial h &=(h\otimes h)\Delta(h^{-1}) \;, && \forall\;h\in \mathbb{G}_1, \\[3pt]
\partial h &=(1\otimes h)(\id\otimes\Delta)(h)(\Delta\otimes\id)(h^{-1})(h^{-1}\otimes 1) \;, && \forall\;h\in \mathbb{G}_2, \\[3pt]
\partial h &=(1\otimes h)(\id\otimes\Delta\otimes\id)(h)(h\otimes 1)
(\Delta\otimes\id^{\otimes 2})(h^{-1})(\id^{\otimes 2}\otimes\Delta)(h^{-1})
\;. && \forall\;h\in \mathbb{G}_3.
\end{align*}

\medskip

If $H$ is commutative, it is not difficult to prove that $(\mathbb{G}_\bullet,\partial)$ is a cochain complex of Abelian groups,
so that the cohomology groups are well-defined. Indeed for $H$ commutative, if $h$ is an $n$-cochain,
$\partial h=\prod_{i=0}^{n+1}\Delta_i(h^{(-1)^i})$ is a group homomorphism and
$$
\partial^2 h=\prod_{i=0}^{n+2}\prod_{j=0}^{n+1}\Delta_i\Delta_j(h^{(-1)^{i+j}})=1 \;,
$$
where the latter equality follows from the relation for face operators: $\Delta_i\Delta_j=\Delta_{j+1}\Delta_i\;\forall\;i\leq j$.

For a general non-commutative Hopf algebra $H$, not only $\partial^2$ is not $1$, but $\partial$ is
not even a group homomorphism. Nevertheless the first cohomology group is well defined.

\begin{lemma}\label{eq:Hzero}
For any Hopf algebra $H$, the first cohomology group is well defined and given by the group of grouplike elements of $H$:
$\{h\in H:\Delta(h)=h\otimes h\}$.
\end{lemma}

\noindent
One can also check that $B_2$ is a subset of $Z_2$, although in general not a subgroup.

\begin{lemma}\label{lemma:2.2}
For any Hopf algebra $H$:
\begin{list}{\quadratini}{\leftmargin=2em \itemsep=2pt \labelsep=6pt}
\item $\partial^2h =1^{\otimes 3}$ for any $h\in \mathbb{G}_1$;
\item $\partial^2h =1^{\otimes 4}$ for any \emph{invariant} $h\in \mathbb{G}_2$.
\end{list}
\end{lemma}

\begin{proof}
From \eqref{eq:DEsubA} and \eqref{eq:DEsubB} and the
coassociativity of $\Delta$, for all $h\in\mathbb{G}_1$:
$$
\partial^2h=
h_2h_3(h^{-1})_{(23)}
h_1h_{(23)}(h^{-1})_{(123)}
h_{(123)}(h^{-1})_3(h^{-1})_{(12)}
h_{(12)}(h^{-1})_2(h^{-1})_1 \;.
$$
After obvious simplifications
$
\partial^2h=
h_2h_3(h^{-1})_{(23)}
h_1h_{(23)}(h^{-1})_3(h^{-1})_2(h^{-1})_1 ,
$
but $h_1$ and $h_{(23)}$ commute, and also $h_1,h_2,h_3$ are mutually commuting, hence the first statement.

The second statement easily follows from the explicit expression
\begin{multline*}
\partial^2h =h_{34}h_{2(34)}(h^{-1})_{(23)4}(h^{-1})_{23}
            h_{(23)4}h_{1(234)}(h^{-1})_{(123)4}(h^{-1})_{1(23)}h_{23}h_{1(23)}
            \\[5pt]
            (h^{-1})_{(12)3}(h^{-1})_{12}
            h_{(12)3}h_{(123)4}(h^{-1})_{(12)(34)}(h^{-1})_{34}
            h_{12}h_{(12)(34)}(h^{-1})_{1(234)}(h^{-1})_{2(34)} \;,
\end{multline*}
by noticing for example that by invariance $(h^{-1})_{(23)4}(h^{-1})_{23}h_{(23)4}=(h^{-1})_{23}$
and so on.
\end{proof}

\noindent
It is not true in general that $\partial^2h =1^{\otimes 4}$ for arbitrary $2$-cochains.
Here is a counterexample.

\begin{ex}\label{ex:2.7}
Let $a,b,c\in H$ be group-like elements with $ab=cba$, $ac=ca$ and $bc=cb$. Then
$$
\partial^2(a\otimes b)=1 \otimes c\otimes c\otimes 1 \;. \vspace{-20pt}
$$
\end{ex}

\vspace{10pt}

\noindent
In particular if $H=\C[SU(2)]$ is the group algebra of $SU(2)$, with linear basis
$\{\delta_g:g\in SU(2)\}$, and $\sigma_1,\sigma_2$ are the first
two Pauli matrices, then $a=\delta_{\sigma_1}$, $b=\delta_{\sigma_2}$
and $c=\delta_{-1}$ satisfy the condition in the example above.
Another concrete example is given by $H=U(\mathfrak{h}_3(\R))\lbrak \hbar\rbrak $, generated by Lie
algebra elements $X,Y,Z$ with $Z$ central and $[X,Y]=Z$; then the triple $a=e^{\hbar X}$,
$b=e^{\hbar Y}$ and $c=e^{\hbar^2Z}$ satisfies the condition in the example above.

\medskip

One can define up to the second cohomology group (for an arbitrary Hopf algebra $H$) by focusing on \emph{invariant} cochains.
Let us denote by $\mathbb{G}_n^{\mathrm{inv}}$ the subgroup of elements of $\mathbb{G}_n$
which are $H$-invariant, i.e.~commute with the image of the iterated coproduct $\Delta^n:H\to H^{\otimes n}$.
Let $Z_n^{\mathrm{inv}}$ be the set of invariant $n$-cocycles and $B_n^{\mathrm{inv}}=\{\partial h: h\in
\mathbb{G}_{n-1}^{\mathrm{inv}}\}$. Then:

\begin{lemma}
$\partial$ maps $\mathbb{G}_n^{\mathrm{inv}}$ into $\mathbb{G}_{n+1}^{\mathrm{inv}}$.
Moreover, $Z_2^{\mathrm{inv}}$ is a group and $B_2^{\mathrm{inv}}$ a normal subgroup.
\end{lemma}
\begin{proof}
From coassociativity, $\Delta^{n+1}=\Delta_i\Delta^{n}$ for all $i=1,\ldots,n$. Hence
$$
\Delta^{n+1}(a)\Delta_i(h)=\Delta_i\big(\Delta^{n}(a)h\big)=
\Delta_i\big(h\Delta^{n}(a)\big)=\Delta_i(h)\Delta^{n+1}(a)
$$
for all $a\in H$, $h\in\mathbb{G}_n^{\mathrm{inv}}$ and $i=1,\ldots,n$. Thus $\Delta_i(h)$ is invariant for all $i$,
and so is $\partial h$.

Clearly $\partial:\mathbb{G}_1^{\mathrm{inv}}\to \mathbb{G}_2^{\mathrm{inv}}$ is a group homomorphism,
and $\mathbb{G}_1^{\mathrm{inv}}$ is a group (the set of central grouplike elements of $H$),
hence $B_2^{\mathrm{inv}}$ is a group too.
One easily checks that $\partial_\pm:\mathbb{G}_2^{\mathrm{inv}}\to \mathbb{G}_3^{\mathrm{inv}}$ are group
homomorphism, hence $Z_2^{\mathrm{inv}}\subset \mathbb{G}_2^{\mathrm{inv}}$ is a subgroup,
since the $2$-cocycle condition can be written as $\partial_+h=(\partial_-h^{-1})^{-1}$.
From the second statement in Lemma \ref{lemma:2.2} it follows the inclusion $B_2^{\mathrm{inv}}\subset Z_2^{\mathrm{inv}}$.
Furthermore, for any $a\in \mathbb{G}_1^{\mathrm{inv}}$ (hence in the center of $H$) and $b\in \mathbb{G}_2^{\mathrm{inv}}$,
$(\partial a)b=b(\partial a)$. This proves that $B_2^{\mathrm{inv}}\subset Z_2^{\mathrm{inv}}$ is a normal subgroup.
\end{proof}

It follows from previous lemma that the second cohomology group is well-defined (but higher cohomology groups in general
are still not defined).

\smallskip

In its dual version, chains are given by linear maps from $H^{\otimes n}$ to the base ring and the group operation is given by
the convolution product. In this case, one can define the first two homology groups for an arbitrary (possibly non-cocommutative)
Hopf algebra by considering \emph{lazy chains} (which are the analogue of invariant cochains).
This definition is originally due to Schauenburg, and systematically studied by Bichon, Carnovale, Guillot, Kassel and collaborators
\cite{BC06,BK08,GK10}.

\smallskip

From the above discussion, we see that the natural definition \ref{def:chain} of cohomology complex doesn't work in the non-Abelian case, at least in the context of Hopf algebras.
Focusing on invariant cochains allows to define the first two cohomology groups, but invariant cochains correspond to trivial deformations of Hopf algebras (cf.~\S\ref{sec:cochain}). So, at least if we are interested in deformation theory, 
invariant cochains are not what we are interested in.

We'll see in the next section that the correct framework for deformation theory (not only of Hopf algebras) is provided by non-Abelian cohomology.

\subsection{Non-Abelian cohomology}\label{sec:nonAb}
A textbook reference on non-Abelian cohomology is \cite{Gir71}. The following (simplified) definition is taken from \cite{Oni67}.
A non-Abelian cochain complex $C^\bullet$ consists of the following data: two groups $C^0,C^1$, a pointed set $(C^2,e)$, a map $\partial_1:C^1\to C^2$ sending the unit of $C^1$ to the basepoint $e$ of $C^2$ and two group morphisms $\alpha:C^0\to\mr{Aff}(C^1)$ and $\beta:C^0\to\mr{Aut}(C^2)$ satisfying
$\partial_1\circ\alpha_x=\beta_x\circ\partial_1\;\forall\;x\in C^0$. We will say that the map $\partial_1$ is $C^0$-\emph{equivariant}.
Here $\mr{Aff}(C^1)=C^1\rtimes\mr{Aut}(C^1)$ is the semidirect product of $C^1$ with the group of automorphisms of $C^1$; the first factor acts on itself by left multiplication, and this gives an action of $\mr{Aff}(C^1)$ (hence of $C^0$) on $C^1$. We denote by $\mr{Aut}(C^2)$, instead, the group of permutations of $C^2$ that don't move the basepoint.

From this data one construct a sequence:
\medskip
\begin{center}
\begin{tikzpicture}
\tikzset{>=stealth}

\node (A) {$\{e\}$};
\node[right=of A] (B) {$C^0$};
\node[right=of B] (C) {$C^1$};
\node[right=of C] (D) {$C^2$};
\node[below=of C] (F) {$\mr{Aff}(C^1)$};
\node[right=of D] (E) {$\{e\}$};

\draw[->] (A) -- (B);
\draw[->] (B) --node[below left]{$\alpha$} (F);
\draw[->] (F) --node[right]{$\pi$} (C);
\draw[->] (C) --node[above]{$\partial_1$} (D);
\draw[->] (D) -- (E);
\draw[->,dashed] (B) --node[above]{$\partial_0$} (C);

\end{tikzpicture}
\end{center}
where $\pi(a,b):=a$ for all $a\in C^1$ and $b\in\mr{Aut}(C^1)$, and
$\partial_0:=\pi\circ\alpha$ (note that $\pi$, and then $\partial_0$, are not group homomorphisms).
The cohomology of this complex is defined by:
$$
H^0(C^\bullet):=\alpha^{-1}\big(\mr{Aut}(C^1)\big)
\;,\qquad
H^1(C^\bullet):=\ker\partial_1/\alpha(C^0)
$$
where $\ker\partial_1:=\{x\in C^1:\partial x=e\}$;
$H^0$ is a a group (a subgroup of $C^0$), and $H^1$ a pointed set.

\begin{rem}
From the equivariance condition, $\partial_1\alpha_x(y)=\beta_x\,\partial_1(y)=\beta_x(e)=e$ for all $x\in C^0$ and $y\in\ker\partial_1$. Thus, $\alpha_x(\ker\partial_1)\subset\ker\partial_1$ and the quotient space $H^1(C^\bullet)$ is well defined.
Note also that $H^0(C^\bullet)=\alpha^{-1}(\ker\pi)$ is exactly the kernel of $\partial_0$.

Since $\partial_0x=\alpha_x(e)$ ($e$ is fixed by $\mr{Aut}(C^1)$),
from the equivariance condition we deduce that, for all $x\in C^0$,
$\partial_1\partial_0x=\beta_x\partial_1e=\beta_x(e)=e$. So, the one above is indeed a cochain complex.
\end{rem}

\medskip

\noindent
From now on we will forget about $\partial_0$, and denote $\partial_1$ simply by $\partial$.

The example provided in \cite{Oni67} is given by the {\v C}ech cochain complex of a sheaf of (non-Abelian) groups, with application to the classification of supermanifolds. Here I give two examples: the deformation complexes controlling deformations of associative algebras and Hopf algebras/module algebras.

\begin{ex}[DGLAs]
Let $(L,\de)$ be a differential graded Lie algebra,
over a field $\Bbbk$, and $R=\Bbbk\oplus\mf{m}$ 
a commutative $\Bbbk$-algebra with $\mf{m}$
a nilpotent ideal.
Call $C^0:=\exp(L^0\otimes\mf{m})$ and
$C^i:=L^i\otimes \mf{m}$, for $i=1,2$. Since $L^0$ is a Lie algebra, $C^0$ is a group with BCH multiplication. Moreover, $C^1$ is a vector space, hence a group w.r.t.~the vector sum.

A map $\partial:C^1\to C^2$ is given by
$\partial x=\de x+\frac{1}{2}[x,x]$; its kernel are the solutions of the Maurer-Cartan equation.
Two morphisms $\alpha:C^0\to\mathrm{Aff}(C^1)$ and $\beta:C^0\to\mathrm{Aut}(C^2)$, the former being the group of affine transformations of the vector space $C^1$ and the latter being the group of permutations of the pointed space $C^2$, are given by \cite[Eq.~(3.69) \& (3.78)]{Esp15}:
\begin{align*}
\alpha_g(y) &:=
e^{\mr{ad}_x}(y+\de)-\de=
\sum_{n=0}^\infty\frac{1}{n!}\,\mr{ad}_x^n(y)-
\sum_{n=1}^\infty\frac{1}{n!}\,\mr{ad}_x^{n-1}(\de x) \;,
\\
\beta_g(z) &:=
e^{\mr{ad}_x}(z)=
\sum_{n=0}^\infty\frac{1}{n!}\,\mr{ad}_x^n(z) \;,
\end{align*}
for all $g=\exp x\in C^0$, $y\in C^1$ and $z\in C^2$ (since $\mf{m}$ is nilpotent, $e^{\mr{ad}_x}$ is a finite sum, hence well defined). Here $\mr{ad}_x(y):=[x,y]$ is the adjoint action, $\mr{ad}_x^n=\smash[b]{\underbrace{\mr{ad}_x\circ\ldots\circ\mr{ad}_x}_{n\text{ times}}}$ if $n\geq 1$, and $\mr{ad}_x^0=\id$.

Since $\partial y=\frac{1}{2}[y+\de,y+\de]\;\forall\;y\in C^1$, the condition  $\partial\alpha_g=\beta_g\partial$ is equivalent to
$$
[\alpha_g(y)+d,\alpha_g(y)+d]=
[e^{\mr{ad}_x}(y+\de),e^{\mr{ad}_x}(y+\de)]=
e^{\mr{ad}_x}\big([y+\de,y+\de]\big)=
\beta_g\big([y+\de,y+\de]\big) \;.
$$
But this is just the property of $\,e^{\mr{ad}_x}$ of being a morphism of graded Lie algebras.

Thus, $\partial$ is $C^0$-equivariant and $C^\bullet$ is a non-Abelian cochain complex.
When $(L,\de)$ is the shifted Hochschild complex
(with Gerstenhaber bracket) of an associative algebra $A$, the quotient space $H^1(C^\bullet):=\ker\partial/\exp(C^0)$ is the moduli space of $R$-deformations of $A$ (see e.g.~\cite{Esp15}).
\end{ex}

\begin{ex}[Hopf algebras]
In the notations of \S\ref{sec:Hopf}, let
$C^i:=\mathbb{G}_{i+1}$, for $i=0,1,2$,
and $\partial:C^1\to C^2$ be the map in
\eqref{eq:mappartial}.
Since $\partial_\pm:\mathbb{G}_1\to \mathbb{G}_2$ are group homomorphisms (the diagonal embedding and the coproduct), we can define an action $\alpha$ of $g\in C^0$ on $h\in C^1$ by:
\begin{equation}\label{eq:action}
\alpha_g(h)=(\partial_+g)h(\partial_-g^{-1}) \;.
\end{equation}
Note that
$\alpha_g(h)=(\partial g)(\mr{Ad}_{\partial_-g^{-1}})^{-1}(h)$ where $\partial g\in C^1$ and $\mr{Ad}_{x}h=xhx^{-1}$ is the adjoint action. Since $\mr{Ad}$ is an action by group automorphisms, $\alpha$ maps $C^0$ into $C^1\rtimes\mr{Aut}(C^1)$.

If we call $\beta$ the action of $C^0$ on $C^2$ given by $\beta_g=\mr{Ad}_{g\otimes g\otimes g}$, one can easily verify that
$\partial\circ\alpha_g=\beta_g\circ\partial\;\forall\;g\in C^0$, and we have another example of non-Abelian cohomology.

In this example, 
$H^0(C^\bullet)$ is the same cohomology group of Lemma \ref{eq:Hzero}, while $H^1(C^\bullet)$ is the set controlling (co)associative deformations of Hopf algebras/module algebras, as we will see in the next section.
\end{ex}

\subsection{Hopf cochains and quantization}\label{sec:cochain}
Let $A$ be a $H$-module algebra, i.e.~an associative unital algebra with an action $\az$ of $H$ satisfying
$h\az (ab)=(h_{(1)}\az a)(h_{(2)}\az b)$ and $h\az 1=\epsilon(h)1$ for all $h\in H$ and $a,b\in A$. Here I use Sweedler notation for
the coproduct, $\Delta(h)=h_{(1)}\otimes h_{(2)}$ with summation understood, and denote by $m:A\otimes A\to A$
the multiplication of $A$, $m(a\otimes b):=ab$. The module algebra condition then reads
$$
h\circ m=m\circ\Delta h
$$
for all $h\in H$, where the action of $h$ on $A$ is understood.

Given any $2$-cochain $F\in\mathbb{G}_2$, we can define a new coproduct $\Delta_F$ on $H$ and a new product $\ast_F$ on $A$ by
\begin{equation}\label{eq:deformedprod}
\Delta_F(h):=F\Delta(h)F^{-1} \;,\qquad
m_F(a\otimes b)=a\ast_Fb:=m\circ F^{-1}(a\otimes b) \;,
\end{equation}
for all $h\in H$ and $a,b\in A$, where in the second formula the action of $H^{\otimes 2}$ on $A^{\otimes 2}$ is understood.
\linebreak
Let $H_F$ be $H$ as an algebra, with the same counit but deformed coproduct $\Delta_F$, and let $A_F$ be $A$ as a vector space but with deformed product $m_F$. By construction
$$
h\circ m_F=m_F\circ \Delta_Fh \;,\qquad\forall\;h\in H_F.
$$
Moreover:
\begin{align}
(\id\otimes\Delta_F)\Delta_F(h) &=\Phi_F
(\Delta_F\otimes\id)\Delta_F(h)\Phi_F^{-1} \;,\quad\forall\;h\in H_F \label{eq:DeltaF} \\[3pt]
m_F(m_F\otimes\id) &=m_F(\id\otimes m_F)\Phi_F \;, \label{eq:quasiass}
\end{align}
where $\Phi_F:=\partial F$ is the \emph{coassociator}.

If $F$ is a \emph{counital} $2$-cochain, then 
$H_F$ is a quasi-Hopf algebra in the sense of Drinfeld \cite{Dri90}
and $A_F$ is a (not necessarily associative) unital algebra and a $H_F$-module algebra.\footnote{%
Counitality guarantees that $(\id\otimes\epsilon)\Delta_F=(\epsilon\otimes\id)\Delta_F=\id$ and that the unit $1$ of
$A$ is also a unit of $A_F$.
One can work with bialgebras and quasi-bialgebras as well, since the antipode plays no role in the above construction.
There are several other axioms in the definition of quasi-bialgebra/quasi-Hopf algebras, besides \eqref{eq:DeltaF}, that
we omitted since we are mainly interested in the associativity property. In the Hopf case there is an explicit formula
for the deformed antipode, and it is a general result that being a quasi-Hopf algebra is a property preserved by twisting
with any invertible counital 2-cochain. For the details one can see \cite{Dri90}.}

As shown in Example \ref{ex:2.7}, even if $\Phi_F$ is a coboundary, it doesn't mean that it is a $3$-cocycle.
On the other hand, it is a cocycle in the Hopf cohomology of $H_F$. More precisely:

\begin{lemma}[{\cite[pag.~64-65]{Maj95}}]
Let $\partial_F$ be the coboundary operator constructed with $\Delta_F$ instead of $\Delta$.
Then for all $2$-cochains $F$, $\partial_F(\partial F)=\partial_F\Phi_F=0$.
\end{lemma}

\begin{proof}
One has
$
\Phi_F=F_{23}F_{1(23)}F^{-1}_{(12)3}F^{-1}_{12}
$
and
\begin{align*}
\partial_F\Phi_F &=
\Phi_{234}F_{23}\Phi_{1(23)4}F_{23}^{-1}\Phi_{123}F_{12}(\Phi^{-1})_{(12)34}F_{12}^{-1}F_{34}(\Phi^{-1})_{12(34)}F_{34}^{-1}
 \\ & =
F_{34}F_{2(34)}\big(F^{-1}_{(23)4}(F^{-1}_{23}
F_{23})F_{(23)4}\big)F_{1(234)}F^{-1}_{(123)4}\big(F^{-1}_{1(23)}
(F_{23}^{-1}
\\ & \qquad
F_{23})F_{1(23)}\big)\big(F^{-1}_{(12)3}(F^{-1}_{12}
F_{12})
F_{(12)3}\big)F_{(123)4}F^{-1}_{(12)(34)}\big(F^{-1}_{34}
F_{12}^{-1}F_{34}
\\ & \qquad
F_{12}\big)F_{(12)(34)}F^{-1}_{1(234)}F^{-1}_{2(34)}
F_{34}^{-1}
 \\ \intertext{(simplifying first the terms in parenthesis)}
 &=
F_{34}\Big(F_{2(34)}\big(F_{1(234)}(F^{-1}_{(123)4}
F_{(123)4}F^{-1}_{(12)(34)}F_{(12)(34)})F^{-1}_{1(234)}\big)F^{-1}_{2(34)}\Big)
F_{34}^{-1}
 \\ &=
F_{34}\Big(F_{2(34)}\big(F_{1(234)}F^{-1}_{1(234)}\big)F^{-1}_{2(34)}\Big)F_{34}^{-1}
 \\ &= 1 \;. \qedhere
\end{align*}
\end{proof}

\noindent
The condition $\partial_F\Phi_F=1$ guarantees that the category of modules of $H_F$,
with operation $\otimes$ given by the Hopf tensor product, is a monoidal category.
Specifically, it is exactly Mac Lane's coherence condition, i.e.~the commutativity of the diagram in Fig.~\ref{fig:P} where the coassociator is seen as a module map $(A\otimes B)\otimes C\to A\otimes (B\otimes C)$
for any three objects $A,B,C$ (and with natural isomorphism of vector spaces understood).
Counitality of $F$, and then of $\Phi_F$, implies the commutativity of the triangle diagram in Fig.~\ref{fig:T}, the second condition in the definition of monoidal category (in fact, the horizontal arrow in this case is the identity).

\begin{figure}[t]
  \centering
  \subfloat[Pentagon diagram]{
\begin{small}
\begin{tikzpicture}[scale=0.8]
\tikzset{>=stealth}

\node (P0) at (90:2.8cm) {$((A\otimes B)\otimes C)\otimes D$};
\node (P1) at (90+72:2.5cm) {$(A\otimes B)\otimes (C\otimes D)$};
\node (P2) at (90+2*72:2.5cm) {$\mathllap{A\otimes (B\,\otimes\,}(C\otimes D))$};
\node (P3) at (90+3*72:2.5cm) {$A\otimes ((B \mathrlap{\,\otimes\,C)\otimes D))}$} ;
\node (P4) at (90+4*72:2.5cm) {$(A\otimes (B\otimes C))\otimes D$};

\draw
(P0) edge[->] node[left=1pt] {\raisebox{8pt}{$\Phi_{(12)34}$}} (P1)
(P1) edge[->] node[left=1pt] {$\Phi_{12(34)}$} (P2)
(P3) edge[->] node[above] {$\Phi_{234}$} (P2)
(P4) edge[->] node[right=1pt] {$\Phi_{1(23)4}$} (P3)
(P0) edge[->] node[right=1pt] {\raisebox{8pt}{$\Phi_{123}$}} (P4);
\end{tikzpicture}
\end{small}
  \label{fig:P} }
	\hspace{8mm}
  \subfloat[Triangle diagram]{
\begin{tikzpicture}
\tikzset{>=stealth}

\node (P0) at (0,0) {$(A\otimes \Bbbk)\otimes B$};
\node (P1) at (3.6,0) {$A\otimes (\Bbbk\otimes B)$};
\node (P2) at (1.8,-2) {$A\otimes B$};
\node (P3) at (1.8,-3) {};

\draw
(P0) edge[->] node[above] {$\Phi$} (P1)
(P0) edge[->] (P2)
(P1) edge[->] (P2);
\end{tikzpicture}
  \label{fig:T} }
	\vspace{-5pt}\caption{}
\end{figure}

\smallskip

With a straightforward computation one proves that $\Delta_F$ is coassociative if and only if the element
\begin{equation}\label{eq:fakecob}
(\partial_-F^{-1})(\partial_+F)
\end{equation}
is $H$-invariant, or equivalently $\Phi_F$ is invariant in $H_F$, i.e.~commutes with the image of the iterated deformed coproduct $\Delta^3_F$.
In this case, we will say that $F$ is an \emph{equivariant twist}.
Note that in general \eqref{eq:fakecob} is not the coboundary of a $2$-cochain like $\Phi_F$ (it is the coboundary of $F$
in the Hopf algebra cohomology of $H^{\mathrm{op}}$, that is $H$ with opposite product).

The module algebra $A_F$ is associative if{}f it is in the kernel of $\Phi_F-1$, which is not guaranteed by the invariance
condition and must be verified case by case. We refer to an algebra with multiplication satisfying \eqref{eq:quasiass} as \emph{quasi-associative}.

Module algebras $A_F$ are automatically associative if we impose the stronger condition $\Phi_F=1$, i.e.~that $F$ is a $2$-cocycle.
Since $\Phi_F=1$ implies that \eqref{eq:fakecob} is also $1$, clearly every cocycle twist is also an equivariant twist.
A summary is in Table \ref{table}.

When $F$ is a cocycle, $H_F$ is a Hopf algebra (not just a bialgebra) with a deformed antipode whose explicit formula can be found in \cite[\S2.3]{Maj95}.

\bigskip

\begin{table}[h]
\centering
\begin{tabular}{ccc}
\hline
\rule{0pt}{14pt}
$F$ & $H_F$ & $A_F$ \\[2pt]

\hline

\rule{0pt}{14pt}
Counital $2$-cochain & quasi-Hopf algebra & unital quasi-associative algebra \\[2pt]

Counital equivariant twist & Hopf algebra & unital quasi-associative algebra \\[2pt]

Counital cocycle twist & Hopf algebra & unital associative algebra \\[2pt]

\hline
\end{tabular}
\smallskip
\caption{}\label{table}\vspace{-4mm}
\end{table}

\begin{rem}
Let $A$ be a $H$-module algebra. If $F$ and $F'=\alpha_g(F)$ are in the same orbit for the action \eqref{eq:action}, for some $g\in\mathbb{G}_1$,
they define isomorphic quasi-Hopf algebras/module algebras: indeed, as one can easily check, 
the map $h\mapsto ghg^{-1}$ is an isomorphism $H_F\to H_F'$
and the map $a\mapsto g\az a$ is an isomorphism $A_F\to A_{F'}$ (see e.g.~Prop.~2.3.3 of \cite{Maj95}).

In particular, Hopf algebra deformations resp.~associative module algebras deformations are parametrized by the cohomology set 
$H^1(C^\bullet):=\ker\partial/\alpha(C^0)$
of \S\ref{sec:nonAb}.
\end{rem}

\begin{ex}[Abelian twist]\label{ex:2.13}
Let $P_1,P_2$ be primitive elements of a Hopf algebra $H$, i.e.~satisfying
$$
\Delta(P_i)=P_i\otimes 1+1\otimes P_i \;,\qquad
\epsilon(P_i)=0 \;,\qquad
S(P_i)=-P_i \;.
$$
Assume that $[P_1,P_2]=0$ and let $H\lbrak \hbar\rbrak $ be the Hopf algebra over the ring $R:=\C\lbrak \hbar\rbrak $ given by formal power series with coefficients in $H$, and operations extended $R$-linearly
(here $\otimes$ is the completed tensor product over $R$).
A counital $2$-cocycle is given by:
\begin{equation}\label{eq:Moyaltwist}
F:=e^{i\hbar P_1\otimes P_2} \;.
\end{equation}
This is cohomologous to
\begin{equation}\label{eq:coctwo}
F':=e^{\frac{i\hbar}{2}(P_1\otimes P_2-P_2\otimes P_1)} \;,
\end{equation}
that is: $F'=(\partial g)F$ where $g:=\exp\big\{\tfrac{i\hbar}{2}P_1P_2\big\}$.
Since $H$ is commutative, this is equivalent to
$F'=\alpha_g(F)$ where $\alpha$ is the action \eqref{eq:action}.
\end{ex}

\medskip

Recall that, for elements $X,Y$ of a Lie algebra with central commutator, Baker-Campbell-Hausdorff formula gives:
\begin{equation}\label{eq:BHC}
e^{\hbar X}e^{\hbar Y}=e^{\frac{1}{2}\hbar^2[X,Y]}e^{\hbar (X+Y)} \;.
\end{equation}

\begin{es}
Let $F,F',g$ be the cochains in Example \ref{ex:2.13}. Using \eqref{eq:BHC} verify that $F$ is a cocycle, and that $F'=\alpha_g(F)$.
\end{es}

\begin{rem}\label{rem:2.15}
Let $P_1,P_2$ be generators of the universal enveloping algebra $H:=\U(\R^2)\lbrak \hbar\rbrak $, acting on $C^\infty(\R^2)$ as derivations:
$$
P_j(f)=-i\partial f/\partial x_j \;,\qquad\forall\;j=1,2.
$$
Then $A:=C^\infty(\R^2)\lbrak \hbar\rbrak $ is a $H$-module algebra. The product \eqref{eq:deformedprod} associated to the cocycle \eqref{eq:coctwo} is the well known \emph{Moyal product} (see e.g.~Example 3.1 of \cite{Esp15}).
\end{rem}

\begin{es}
Who is the Hopf algebra $H_F$ (or $H_{F'}$) in Rem.~\ref{rem:2.15}?
\end{es}

In this section we discussed deformations of (quasi-)Hopf algebras where the coproduct changes (is conjugated to the original one via a $2$-cochain, cf.~\eqref{eq:deformedprod}), but the product doesn't. Moreover, this theory is not specifically designed for formal deformations.
On the other hand, one might be interested in formal deformations where both product and coproduct change.
There is a theory similar to the one of formal deformations of associative or Poisson algebras, where one starts with ``infinitesimal'' deformations. Their existence and the obstruction to extend an order $n$ deformation to an order $n+1$ are controlled as usual by a cohomology theory. The interested reader can find this subject described in details in Chapter 6 of \cite{CP94}.

\subsection{Heisenberg manifolds and the noncommutative torus}

We saw in \S\ref{sec:1} (and in particular, in \S\ref{sec:1.5}), that given a self-Morita equivalence bimodule over a $C^*$-algebra -- or a ``noncommutative line bundle'' over a ``noncommutative space'' $X$ -- we can construct the analogue of a principal $U(1)$-bundle $P\to X$. It would be nice to apply this construction to the noncommutative torus (Example \ref{ex:2.3}), which is probably the best known example of noncommutative space. Unfortunately, it has non-trivial SMEBs only for special values of $\theta$ (it must be a real quadratic irrationality, corresponding to a noncommutative torus with ``real multiplication'' \cite{Man04}).
This difficulty can be overcome by using quasi-associative algebras, i.e.~monoid objects \cite{ML98} in the category of representations of a quasi-Hopf algebra. This is a mild kind of non-associativity: as explained in \cite{AM98}, working with such objects is not very different than working with associative algebras.

In this section, I will explain how to deform a principal $U(1)$-bundle $M_3\to\T^2$ by means of cochain quantization and obtain a $\Z$-graded quasi-associative algebra $A=\bigoplus_{n\in\Z}A_n$ whose weight spaces are the well known Connes-Rieffel imprimitivity bimodules (``line bundles'') over the algebra $A_0$ of the noncommutative torus. For simplicity, I will work with formal deformations. Material in this section is taken from \cite{DF15}.

\subsubsection{The Heisenberg manifold $M_3$}

Let $H_3(\R)$ be the group of $3\times 3$ matrices
$$
(x,y,t):=
\left[\!\begin{array}{ccc}
1 & x & t \\ 0 & 1 & y \\ 0 & 0 & 1
\end{array}\!\right] \;,\qquad x,y,t\in\R.
$$
Right invariant vector fields (commuting with the right regular action, hence obtained by
differentiating the left regular one) have a basis of three elements
$$
X:=\frac{\partial}{\partial x}+y\frac{\partial}{\partial t}
\;,\qquad
Y:=\frac{\partial}{\partial y}
\;,\qquad
T:=\frac{\partial}{\partial t}
\;.
$$
which generate the universal enveloping algebra of the Heisenberg Lie algebra $\U(\mf{h}_3(\R))$.
Note that $T$ is central and $[X,Y]=-T$.

Let $H_3(\Z):=\{(x,y,t)\in H_3(\R):x,y,t\in\Z\}$.
By right invariance, these vector fields descend to the
$3$-dimensional Heisenberg manifold $M_3:=H_3(\R)/H_3(\Z)$, which is the total space of a principal $U(1)$-bundle over the torus $\T^2$. 
Thinking of functions on $G/K$ as right $K$-invariant functions on $G$, we get:
$$
C^\infty(M_3)=\big\{f\in C^\infty(S^1\times\R\times S^1):f(x,y+1,t+x)=f(x,y,t)\big\} \;,
$$
where $S^1=\R/\Z$ (so $f\in C^\infty(M_3)$ is periodic with period $1$ in $x$ and $t$).
The action of central elements $(0,0,t)\in H_3(\R)$ descends to a principal action of $U(1)$ on $M_3$, and $M_3/U(1)\simeq\T^2$. We identify
$C^\infty(\T^2)$ with the subset of $f\in C^\infty(M_3)$ that do not depend on $t$ (and so are periodic in both $x$ and $y$):
$$
C^\infty(\T^2)=\big\{f\in C^\infty(M_3):f(x,y,t+t')=f(x,y,t)\;\forall\;x,y,t,t'\big\} \;.
$$
By standard Fourier analysis every $f\in C^\infty(M_3)$ is of the form:
\begin{equation}\label{eq:fxyt}
f(x,y,t)=\sum_{m,n\in\Z}e^{2\pi i(mx+nt)}\widehat{f}(y;m,n) \;,
\end{equation}
where $\widehat{f}:\R\times\Z^2\to\C$ is a function satisfying
$$
\widehat{f}(y+1;m,n)=\widehat{f}(y;m+n,n)
$$
Note that $T$ has discrete spectrum when acting on $C^\infty(M_3)$ (by periodicity in $t$). For $n\in\Z$, if
$$
L_n:=\big\{f\in C^\infty(M_3):Tf=(2\pi in)f\big\}\;,
$$
then the algebra $\bigoplus_{n\in\Z}L_n$ is dense in $C^\infty(M_3)$.
Elements of $L_n$ are smooth sections of a line bundle of degree $n$ on $\T^2$ (these are all smooth line bundles on $\T^2$, modulo isomorphisms). In particular, $L_0=C^\infty(\T^2)$.

For all $n\neq 0$, there is a bijection $f\in L_n\mapsto\widetilde{f}\in\mc{S}(\R\times\Z/n\Z)$ (where $\mc{S}$ denotes the set of Schwartz functions), given by \cite{DFF13}:
\begin{equation}\label{eq:fxy}
f(x,y,t)=\sum_{m\in\Z}e^{2\pi i(mx+nt)}\widetilde{f}(y+\tfrac{m}{n};m) \;.
\end{equation}
This is called Weil-Brezin-Zak transform in solid state physics \cite[\S1.10]{Fol89}.

\subsubsection{A ``principal $U(1)$-bundle'' over the noncommutative torus}\label{sec:2.6.2}
Define the parameter space:
\begin{equation}\label{eq:parspace}
\Theta:=\hbar\,\C\lbrak\hbar\rbrak \;.
\end{equation}
Let $\theta\in\Theta$.
We now deform the pointwise product of $A=C^\infty(M_3)\lbrak \hbar\rbrak $ with the cochain:
$$
F_\theta=e^{\frac{i\theta}{2\pi}X\otimes Y}
$$
based on the Hopf algebra $H:=\U(\mf{h}_3(\R))\lbrak \hbar\rbrak $. Note the similarity with \eqref{eq:Moyaltwist}.

\begin{es}\label{ex:2.19}
Let $\{a_n,b_n\}_{n\geq 0}$ belong to an associative algebra. 
Prove that the expression
$\sum_{n\geq 0}a_n\big(\sum_{k\geq 0}b_k\hbar^k\big)^n$ is a well defined formal power series (i.e.~that for each $N\geq 0$ the coefficient of $\hbar^N$ is a finite sum).
Realize that $F_{\theta}$ is a well defined power series of $\hbar$.
\end{es}

\noindent
Let $a\ast_\theta b:=m\circ F_\theta^{-1}(a\otimes b)$ be the product in \eqref{eq:deformedprod}, and $A^{\theta}=(A,\ast_\theta)$ the associated algebra. The star-product preserves the vector space decomposition
$A=\overline{\bigoplus_{n\in\Z}A_n}$, where
$A_n:=L_n\lbrak\hbar\rbrak $ (since $T$ is central, each set $L_n$ is stable under the action of $X,Y$). So, $A^{\theta}$ has a dense $\Z$-graded subalgebra.

\begin{lemma}\label{gcl}
For all $\theta,\theta'\in\Theta$:
\begin{equation*}
(\Delta\otimes\id)(F_\theta^{-1})(F_{\theta'}^{-1}\otimes 1)
=(\id\otimes\Delta)(F_{\theta'}^{-1})(1\otimes F_\theta^{-1})\Phi_{\theta,\theta'}
\end{equation*}
where
$$
\Phi_{\theta,\theta'}:=\exp\left\{-\frac{i}{2\pi}X\otimes\left(\theta-\theta'+\frac{i}{2\pi}\theta\theta'\,T\right)\otimes Y\right\} \;.
$$
\end{lemma}

\begin{proof}
The proof is a direct computation using 
Baker-Campbell-Hausdorff formula \eqref{eq:BHC}.
\end{proof}

\begin{rem}
It is straightforward to verify that the coassociator $\Phi_{\theta,\theta}=e^{(\theta/2\pi)^2X\otimes T\otimes Y}$ is not trivial nor invariant,
so that $F_\theta$ is neither a cocycle nor an equivariant twist.
\end{rem}

\medskip

\noindent
There is an action $\alpha$ of $\Z$ by automorphisms on the vector space $\Theta$ given by
\begin{equation}\label{eq:nact}
\alpha_k(\theta)=\frac{\theta}{1+k\theta}=
\theta-k\theta^2+\ldots
 \;,\qquad k\in\Z,\theta\in\Theta\;.
\end{equation}
Clearly $\alpha$ maps $\Theta$ into itself, and $\alpha_j\alpha_k=\alpha_{j+k}\;\forall\;j,k\in\Z$.

\begin{prop}\label{prop:gcl}
The generalized associativity law
\begin{equation}\label{eq:genass}
(a\ast_{\theta'}b)\ast_\theta c=
a\ast_{\theta'}(b\ast_\theta c)
\end{equation}
holds for all $a,c\in A$ and $b\in A_n$
if and only if $\theta'=\alpha_n(\theta)$, with $\alpha$ as in \eqref{eq:nact}.
\end{prop}
\begin{proof}
From the definition of star product and the observation that $Tb=2\pi inb$, hence $\Phi_{\theta,\theta'}$ is the identity on $a\otimes b\otimes c$.
\end{proof}

\noindent
Many properties of $A^{\theta}$ can be deduced from Prop.~\ref{prop:gcl}.

\begin{thm}~\setcounter{enumi}{1}
\begin{list}{\arabic{enumi})\stepcounter{enumi}}{\leftmargin=2em \itemindent=-1em \itemsep=5pt}
\item
$A_0^{\theta}=(A_0,\ast_\theta)$ is an associative unital subalgebra of $A^{\theta}$.

\item
$A_n$ is an $A_0^{\theta'}$-$A_0^{\theta}$-bimodule, with $\theta'=\alpha_n(\theta)$.

\item
$\,\ast_\theta:A_j\otimes A_k\to A_{j+k}$ descends to a map
$A_j\otimes_{\smash[t]{A_0^{\theta}}}\! A_k\to A_{j+k}$ (where the left and right module structure are given by the $\ast_\theta$ multiplication).

\item
For all $m,n,p\in\Z$ and
$\theta'=\alpha_n(\theta)$, the following diagram commutes:

\vspace{7pt}

\begin{center}
\begin{tikzpicture}
\tikzset{>=stealth, shorten >=1pt, shorten <=1pt}

\node (A) at (0,0) {$A_m\otimes_{\smash[t]{A_0^{\theta'}}}\! A_n\otimes_{\smash[t]{A_0^{\theta}}} A_p$};
\node (B) at (-1.9,-1.9) {$A_m\otimes_{\smash[t]{A_0^{\theta'}}}\! A_{n+p}$};
\node (C) at (1.9,-1.9) {$A_{m+n}\otimes_{\smash[t]{A_0^{\theta}}} A_p$};
\node (D) at (0,-3.8) {$A_{m+n+p}$}; 

\draw[->] (A) --node[left]{\raisebox{9pt}{\footnotesize $\id\otimes\ast_\theta$}} (B);
\draw[->] (A) --node[right]{\raisebox{9pt}{\footnotesize $\ast_{\theta'}\otimes\id$}} (C);
\draw[->] (B) --node[below left=-3pt]{{\footnotesize $\ast_{\theta'}\;$}} (D);
\draw[->] (C) --node[below right=-2pt]{{\footnotesize $\;\ast_\theta$}} (D);

\end{tikzpicture}
\end{center}
\end{list}
\end{thm}

\noindent
It is not difficult to verify that $A^\theta_0$ is generated by two unitary elements, the functions
$$
U(x,y,t):=e^{2\pi ix} \;,\qquad
V(x,y,t):=e^{2\pi iy} \;,
$$
with relation
$$
U\ast_\theta V=e^{2\pi i\theta}\,V\ast_\theta U \;.
$$
By point (1) of the above theorem, $A_0^\theta$ is the formal analogue of the algebra in Example \ref{ex:2.3}.
Point (2) is the analogue of e.g.~Eq.~(2.6) of \cite{Pla06}, stating that the algebra of endomorphisms of a finitely generated projective module over a noncommutative torus is another noncommutative torus with different deformation parameter.
Point (4) is the formal version of \cite[Prop.~1.2(b)]{PS02}.

\subsubsection{Comparison with the literature}
Finitely generated projective modules for the noncommutative torus have been introduced in \cite{Con80,Rie83}. Here I will adopt notations similar to \cite{Pla06}. One can compare the bimodules $A_n$ with Connes-Rieffel modules using the Weil-Brezin-Zak transform.
In the notations of previous section, 
we can extend $\C\lbrak \hbar\rbrak $-linearly the map $f\mapsto\widetilde{f}$ in \eqref{eq:fxy} to a map $A_n\to E_n:=\mc{S}(\R\times\Z/n\Z)\lbrak\hbar\rbrak $.
Let $\theta'=\alpha_n(\theta)$ as in previous section.
We define a left action of $A^{\theta'}_0$
on $E_n$ and a right action of $A^\theta_0$ on $E_n$ by:
\begin{equation}\label{eq:bimod}
a.\widetilde{f}:=\widetilde{a\ast_{\theta'}f}
\qquad\quad
\widetilde{f}.a:=\widetilde{f\ast_\theta a}
\end{equation}
for all $\widetilde{f}\in E_n$ and $a\in A_0$.

A straightforward computation using 
\eqref{eq:fxy} gives, for all $\widetilde f\in E_n$ ($n\neq 0$):
\begin{align*}
(U.\widetilde{f})(y;m) &=\widetilde{f}(y+\theta'-\tfrac{1}{n};m-1) \;,&
(\widetilde{f}.U)(y;m) &=\widetilde{f}(y-\tfrac{1}{n};m-1) \;, \\[2pt]
(V.\widetilde{f})(y;m) &=e^{2\pi i(y-\frac{m}{n})}\widetilde{f}(y;m) \;, &
(\widetilde{f}.V)(y;m) &=e^{2\pi i(y-\frac{m}{n})}e^{2\pi in\theta y}\widetilde{f}(y;m) \;,
\end{align*}
where for a smooth function $\psi$,
by $\psi(y+\theta')$ we mean the formal power series
$$
\psi(y+\theta'):=e^{\theta'\partial_y}\psi(y)=\sum\nolimits_{k\geq 0}\tfrac{\theta'^k}{k!}\partial_y^k\psi(y) \;.
$$
These equations should be compared with, e.g., Eq.~(2.1-2.5) of \cite{Pla06}.\footnote{Here we are considering the case \mbox{$g=\,${\tiny\setlength{\arraycolsep}{2pt}$\bigg(\begin{array}{cc} 1 & 0 \\[-1pt] n & 1 \end{array}\bigg)$}}, in the notations of \cite[(2.3)]{Pla06}.
In order to get exactly the same formulas as in \cite{Pla06}
one should replace \eqref{eq:bimod} by
the definition $a.\widetilde{f}:=\widetilde{f\ast_{-\theta'}\hspace{-1pt}a}$
and
$\widetilde{f}.a:=\widetilde{a\ast_{-\theta}\hspace{-1pt}f}$.}

\smallskip

Finally, for all $\widetilde f_1\in E_n$ and $\widetilde f_2\in E_p$,
one can compute $\widetilde{f_1\ast_\theta f_2}\in E_{n+p}$
using \eqref{eq:fxy} and find the explicit formula:
$$
\widetilde{f_1\ast_\theta f_2}(y;m)=\sum_{j+k=m}\widetilde{f}_1\big(y+\tfrac{jp-kn}{n(n+p)};j\big)\widetilde{f}_2\big((1+n\theta)y-
(1-p\theta)\tfrac{jp-kn}{p(n+p)};k\big) \;,
$$
which is valid for $n,p,n+p\neq 0$. This is the analogue of \cite[Prop.~1.2(a)]{PS02}.\footnote{One gets the same notations of \cite{PS02} by defining the pairing of bimodules as the map $
(\widetilde f_1,\widetilde f_2)\mapsto
\widetilde{f_2\ast_{-\theta}\!f_1}$.}

\smallskip

For a discussion about complex structures on (formal) noncommutative tori (quantum theta functions, etc.), see \cite{DF15}.

\subsubsection{Some remarks on formal deformations of line bundles}\label{sec:2.6.4}
A general study of formal deformations of line bundles can be found in \cite{BW00,Bur02,BW02} for symplectic manifolds and \cite{BDW12} for a general Poisson manifold.
Given a smooth vector bundle $E\to X$, one can consider the corresponding $C^\infty(X)$-module of smooth sections $\Gamma^\infty(E)$. If $A_\star=(C^\infty(X)\lbrak\hbar\rbrak,\star)$ is a deformation quantization, one may show that there always exists a right $A_\star$-module structure $M\times A_\star\to M$ on
$M:=\Gamma^\infty(E)\lbrak\hbar\rbrak$: there is no obstruction for deformations of such modules \cite{BW00}. The algebra $\mr{End}_{A_\star}(M)$
is a deformation quantization of $\mr{End}_{C^\infty(X)}(\Gamma^\infty(E))$.
In particular, if $E\to X$ is a line bundle then
$\mr{End}_{C^\infty(X)}(\Gamma^\infty(E))\simeq C^\infty(X)$,
and $\mr{End}_{A_\star}(M)\simeq (C^\infty(X)\lbrak\hbar\rbrak,\star')=:A_{\star'}$ provides another deformation quantization of $X$, and $M$ is a Morita equivalence $A_{\star'}$-$A_\star$ bimodule (similarly to the example of the noncommutative torus). If $\mr{Def}(X)$ denotes the set of equivalence classes of star products on $X$, the map $\star\mapsto\star'$ induces a well-defined action of the Picard group $\mr{Pic}(X)$ on $\mr{Def}(X)$ \cite{Bur02}. In fact, if $\mr{Def}(X,\pi)$ denotes the set of equivalence classes of deformations in the direction of a fixed Poisson structure $\pi$ on $X$, one may show that the above action gives by restriction an action of $\mr{Pic}(X)$ on $\mr{Def}(X,\pi)$. As shown in \cite{Bur02}, the quotient $\mr{Def}(X)/\mr{Pic}(X)$ (resp.~$\mr{Def}(X,\pi)/\mr{Pic}(X)$) is the moduli space of Morita equivalent star products on $X$ (resp.~deformation quantizations in the direction of $\pi$).

\smallskip

For $X=\T^2$ with standard symplectic structure, star products are classified by
$$
\hbar\hspace{1pt}H^2_{\mr{deRham}}(X,\C)\lbrak\hbar\rbrak\simeq \hbar\,\C\lbrak\hbar\rbrak=\Theta \;,
$$
i.e.~the parameter space in \eqref{eq:parspace},
and the action of $\mr{Pic}(X)\simeq\Z$ is the one in \eqref{eq:nact}.

\smallskip

It would be interesting to understand if the ``principal bundle'' construction of \S\ref{sec:2.6.2} can be extended to more general symplectic or Poisson manifolds.

\medskip

{\noindent\bf Acknowledgements.}
I would like to thank Chiara Esposito and Stefan Waldman (the organizers of the school) for the invitation and all the participants for the stimulating discussions and the friendly atmosphere. Special thanks go to Henrique Bursztyn for his remarks and comments, which led to the drafting of \S\ref{sec:2.6.4}.


\end{document}